\documentclass[journal]{IEEEtran}

\input{xhc.sty}

\begin{document}

\title{Data-driven Coordination of Distributed Energy Resources for Active Power Provision}

\author{Hanchen~Xu,~\IEEEmembership{Student Member,~IEEE,}
        Alejandro~D.~Dom\'{i}nguez-Garc\'{i}a,~\IEEEmembership{Member,~IEEE}
        and Peter~W.~Sauer,~\IEEEmembership{Life~Fellow,~IEEE}
\thanks{This work has been funded in part by the Department of Energy under a subcontract from the University of Tennessee Knoxville, the Siebel Energy Institute, and the Advanced Research Projects Agency-Energy (ARPA-E) within the NODES program, under Award DEAR0000695.}
\thanks{The authors are with the Department of Electrical and Computer Engineering at the University of Illinois at Urbana-Champaign, Urbana, IL 61801, USA. Email: \{hxu45, aledan, psauer\}@illinois.edu.}
}

\maketitle

\begin{abstract}
In this paper, we propose a framework for coordinating distributed energy resources (DERs) connected to a power distribution system, the model of which is not completely known, so that they collectively provide a specified amount of active power to the bulk power system, while respecting distribution line capacity limits.
The proposed framework consists of (i) a linear time-varying input-output (IO) system model that represents the relation between the  DER active power injections (inputs),  and the  total active power exchanged between the distribution and bulk power systems (output);   (ii) an estimator that aims to estimate the IO model parameters, and (iii) a controller that determines the optimal DER active power injections so the power exchanged between both systems equals to the specified amount at a minimum generating cost.

We formulate the estimation problem as a box-constrained quadratic program and solve it using the projected gradient descent algorithm.
To resolve the potential issue of collinearity in the measurements, we introduce random perturbations in the DER active power injections during the estimation process.
Using the estimated IO model, the optimal DER coordination problem to be solved by the controller can be formulated as a convex optimization problem, which can be solved easily.
The effectiveness of the framework is validated via numerical simulations using the IEEE 123-bus distribution test feeder.
\end{abstract}


\section{Introduction} \label{sec:intro}

\IEEEPARstart{I}n the modernization that electric power systems are currently undergoing, one goal is to massively integrate distributed energy resources (DERs) into power distribution systems \cite{driesen2008design}.
These DERs, which include distributed generation resources, energy storage, demand response resources, and typically have  small capacities, may be coordinated so as to collectively provide grid support services, e.g., reactive power support for voltage control \cite{kulmala2014voltage, xu2018voltage, zhang2018dynamic}, and active power control for frequency regulation \cite{adg2010der, guggilam2017regulation, xu2018frequency}. 

In this paper, we focus on the problem of coordinating the response of a set of DERs in a lossy power distribution system so that they collectively provide some amount of active power to the bulk power system, e.g., in the form of demand response.
Specifically, the DERs will be requested  to collectively provide in real-time a certain amount of active power at the bus where the power distribution system is interconnected with the bulk power system.
In order for the DERs to fulfill such request, it is necessary to develop appropriate schemes that explicitly take into consideration the losses incurred. 
One way to include the losses in the problem formulation is to utilize a power-flow-like model of the power distribution system obtained offline. 
However, such a model may not be available or if so, it may not be accurate \cite{ardakanian2017identification, chen2014sensitivity, chen2016jacobian}.

As an alternative to the aforementioned model-based approach,  data-driven approaches have been demonstrated to be very effective in such situations where models are not readily available \cite{hou2011data, zhang2016sensitivity, chen2014sensitivity, chen2015sparse, chen2016jacobian, horn2016sced, lu2015wams, zhang2016damping}.
The fundamental idea behind data-driven approaches is to describe the system behavior by a linear time-varying (LTV) input-output (IO) model, and estimate the parameters of this model via regression using measurements of pertinent variables \cite{hou2011data, zhang2016sensitivity}.
Many previous works have applied  data-driven approaches to power system problems, both in a steady-state setting \cite{chen2014sensitivity, chen2015sparse, chen2016jacobian, horn2016sced}, and a dynamical setting \cite{lu2015wams, zhang2016damping}.
For example, in  \cite{chen2014sensitivity},  the authors developed a data-driven framework to estimate linear sensitivity distribution factors such as injection shifting factors  \cite{chen2015sparse}; they further proposed a data efficient sparse representation to estimate these sensitivities \cite{chen2015sparse}.
This framework was later tailored to the problem of estimating  the power flow Jacobian  \cite{chen2016jacobian}. 
In \cite{horn2016sced}, the authors used the estimation framework proposed in \cite{chen2014sensitivity}  to solve the security constrained economic dispatch problem.  
Data-driven approaches have also been applied successfully to develop power systems stabilizers \cite{lu2015wams}, and damping controls \cite{zhang2016damping}.
We refer interested readers to \cite{hou2017overview} for an overview of data-driven approaches and their applications in a variety of other areas. 

Yet, due to the collinearity in the measurements \cite{zhang2017sensitivity}, the regression problem may be ill-conditioned, thus resulting in large estimation errors \cite{zhang2017noise}.
Though numerical approaches such as locally weighted ridge regression \cite{zhang2017sensitivity} and noise-assisted ensemble regression \cite{zhang2017noise} can be used to mitigate the impacts of collinearity, there is no theoretical guarantee on the estimation errors.
In this paper, we pursue the data-driven approach to develop a framework for coordinating the response of a set of DERs.
The proposed framework consists of three components, namely (i) a model of the system describing the relation between the variables of interest to the problem, i.e., DER active power injections and power exchanged between the distribution and bulk power systems, (ii)  an estimator, which provides estimates of the parameters that populate the model in (i); and (iii)  a controller that uses the model in (i) with the  parameters estimated via (ii) to determine the active power injection set-points of the DERs.
Specifically,  an LTV IO model is adopted as the system model to capture the relation between the DER active power injections (inputs),  and the  total active   power exchange  (output).  

The parameters in this model are estimated by the estimator via the solution of a box-constrained quadratic program, obtained by using the projected gradient descent (PGD) algorithm. 
Inspired by ideas in power system model identification \cite{zhang2016mpm, zhang2017online}, we introduce random perturbations in the DER active power injections during the estimation process to resolve the potential issue of collinearity in the measurements used by the estimator.
We show that the estimation algorithm converges almost surely (a.s.) under some mild conditions, i.e., the estimated parameters converge to their true values, and the total provided active power also converges to the required amount.
Using the estimated IO model, the optimal DER coordination problem (ODCP) to be solved by the controller can be formulated as a convex optimization problem, which can be solved easily.
The major contributions of this paper are the data-driven coordination framework, the algorithm to solve the estimation problem, and its convergence analysis.

The remainder of this paper is  organized as follows.
Section~\ref{sec:prelim} describes the  power distribution system model and the ODCP of interest.
Section \ref{sec:framework} describes the components of the data-driven DER coordination framework. 
A description of the algorithm used in the framework, as well as its convergence analysis, is provided in Section \ref{sec:algo}.
The proposed framework is illustrated and validated via numerical simulations on the IEEE 123-bus test feeder in Section \ref{sec:simu}.
Concluding remarks are presented in Section \ref{sec:con}.

\section{Preliminaries} \label{sec:prelim}

In this section, we introduce the power distribution system model adopted in this work.
We then discuss the  DER coordination problem of interest.

\subsection{Power Distribution System Model} \label{sec:model}

Consider a power distribution system represented by a directed graph that consists of a set of buses indexed by the elements in the set $\calN = \{0, 1, \cdots, N\}$, and a set of distribution  lines indexed by the elements in the set $\calL = \{1, 2, \cdots, L\}$. 
Each line $\ell \in \calL$ is associated with a tuple $(i, j) \in \calN \times \calN$, where $i$ is the sending end and $j$ is the receiving end of a line, with the direction from $i$ to $j$ defined to be positive.
Assume bus $0$ corresponds to a substation bus, which is the only connection of the power distribution system to the  bulk power system. 
Further, assume that bus~$0$ is an infinite bus that maintains a constant voltage magnitude.
Without loss of generality, assume there is at most one DER and/or load at each bus, except bus $0$, which does not have any DER or load.
Let $\calN^g = \{1, \cdots, n\}$ denote the DER index set.
Throughout this paper, DERs are assumed to be controllable. 
Uncontrollable DERs are modeled as negative loads.

Let $p_i^d$ and $q_i^d$ respectively denote the active and reactive power loads at bus $i$, $i \in \calN$, and define $\bm{p}^d = [p_1^d, \cdots, p_N^d]$, and $\bm{q}^d = [q_1^d, \cdots, q_N^d]^\top$.
Let $p_i^g$ and $q_i^g$ respectively denote the active and reactive power injections from DER $i$, $i \in \calN^g$, and define $\bm{p}^g = [p_1^g, \cdots, p_n^g]^\top$, and $\bm{q}^g = [q_1^g, \cdots, q_n^g]^\top$.
Let $\underline{p}_i^g$ and $\overline{p}_i^g$ respectively denote the minimum and  maximum active  power that can be provided by DER~$i$, $i \in \calN^g$, and define $\bm{\underline{p}}^g = [\underline{p}_1^g, \cdots, \underline{p}_n^g]^\top$, and $\bm{\overline{p}}^g = [\overline{p}_1^g, \cdots, \overline{p}_n^g]^\top$. 
Similarly, let $\underline{q}_i^g$ and $\overline{q}_i^g$ respectively denote the minimum and maximum reactive power that can be provided by DER~$i$, $i \in \calN^g$, and define $\bm{\underline{q}}^g = [\underline{q}_1^g, \cdots, \underline{q}_n^g]^\top$, and $\bm{\overline{q}}^g = [\overline{q}_1^g, \cdots, \overline{q}_n^g]^\top$.
Let $p_i$ denote the active power injection at bus $i\in \calN$, and define $\bm{p} = [p_1, \cdots, p_N]^\top$; then
\begin{align}
	\bm{p} = \bm{C} \bm{p}^g - \bm{p}^d,
\end{align}
where $\bm{C} \in \bbR^{N \times n}$ is the matrix that maps the DER indices to the buses.
The entry at the $i^{\text{th}}$ row, $j^{\text{th}}$ column of $\bm{C}$ is $1$ if DER $j$ is at bus $i$. 

Let $\tilde{\bm{M}}=[M_{i \ell}] \in \real^{(N+1) \times L}$, with $M_{i \ell} = 1$ and $M_{j \ell} = -1$ if line $\ell$ starts from bus $i$ and ends at bus $j$, and all other entries equal to zero.
Let $\bm{M}$ denote the $(N\times L)$-dimensional matrix that results from removing the first row in $\tilde{\bm{M}}$.
We first make the following assumption:
\begin{assumption} \label{a1}
The power distribution system is radial.
\end{assumption}
Under Assumption~\ref{a1}, $L=N$, and $\bm{M}$ is invertible.
Let $f_\ell$ denote the active power that flows from the sending end to the receiving end of line $\ell \in \calL$, and define $\bm{f} = [f_1, \cdots, f_L]^\top$. 
Let  $\overline{f}_\ell$ denote maximum power flows on line $\ell \in \calL$, and define $\overline{\bm{f}} = [\overline{f}_1, \cdots, \overline{f}_L]^\top$.
Then $\bm{f}$ can be approximately computed from $\bm{p}$ as follows:
\begin{align} \label{eq:line_flow}
	\bm{f} \approx \bm{M}^{-1} \bm{p} = \bm{M}^{-1} (\bm{C} \bm{p}^g - \bm{p}^d).
\end{align}
We would like to point out that Assumption~\ref{a1} allows us to approximately compute the active power flow on the distribution lines.
It is, however, not a necessary condition for the analysis to be done in the rest of the paper.

Let $y$ denote the active power exchanged between the distribution and bulk power systems via bus~$0$, defined to be positive if the flow is from the substation to the bulk power system.
Conceptually, $y$ can be represented as a function of $\bm{p}^g$, $\bm{q}^g$, $\bm{p}^d$, $\bm{q}^d$.
Note also $\bm{q}^g$ is typically set according to some specific reactive power control rules to achieve certain objectives such as constant voltage magnitude or constant power factor \cite{xu2018frequency}, and thus is a function of $\bm{p}^g, \bm{p}^d, \bm{q}^d$.
Then, $y$ can be written as a function of $\bm{p}^g, \bm{p}^d, \bm{q}^d$ as follows:
\begin{align} \label{eq:f}
    y = h(\bm{p}^g, \bm{p}^d, \bm{q}^d),
\end{align}
where $h(\cdot)$ captures the impacts from both the physical laws as well as reactive power control.
We emphasize that although the voltage control problem is not explicitly modeled in this paper, we assume certain voltage control schemes exist in the power distribution system such that the voltage profile will be maintained within an acceptable range.
Indeed, voltage control schemes may have a significant impact on $h$. 
We refer interested readers to \cite{xu2018frequency} for a detailed analysis on the impacts of voltage control schemes.

Note that the explicit form of $h$ is difficult to obtain; however, we can make the following assumption on $h$:
\begin{assumption} \label{a2}
	The function $h$ is differentiable and its first order partial derivatives with respect to $\bm{p}^g$ belong to $[\underline{b}_1, \overline{b}_1]$, where $\underline{b}_1$, $\overline{b}_1$ are some known constants.
	In addition, $\frac{\partial h}{\partial \bm{p}^g}$ is a Lipschitz function, i.e., there exists $b_2 > 0$ such that
	\[
		\Norm{\left.\frac{\partial h}{\partial \bm{p}^g}\right\vert_{\bm{a}} - \left.\frac{\partial h}{\partial \bm{p}^g}\right\vert_{\bm{b}}} \leq b_2 \norm{\bm{a} - \bm{b}},
	\]
	where $\bm{a}, \bm{b} \in [\underline{\bm{p}}^g, \overline{\bm{p}}^g]$, and $\Norm{\cdot}$ denotes the $L_2$-norm.
\end{assumption}

Assumption~\ref{a2} implies that, for fixed loads,  the rate of change in $y$ is bounded for bounded changes in the DER   active power injections. 
In addition, the total active power provided to the bulk power system will increase when more active power is injected in the power distribution system.
This assumption holds when the system is at a normal operating condition without line congestions.

\subsection{Optimal DER Coordination Problem}

The DERs in the distribution system can collectively provide active power to the bulk power system as quantified by the power exchange  between both systems at the substation bus.
For example, the DERs can provide demand response services or frequency regulation services to the bulk power system; in both cases, the DERs need to be coordinated in such a way that the total active power provided to the bulk power system, $y$, tracks some pre-specified value, denoted by $y^\star$.
The objective of the ODCP is to determine the DER active power injections, $\bm{p}^g$, that  minimize some cost function, e.g., one that reflects the cost of active power provision, while respecting to the following constraints:
\begin{itemize} 
\item[\textbf{[C1.]}] the active power exchanged between the distribution and bulk power systems via bus~$0$, $y$, tracks some pre-specified value $y^\star$;
\item[\textbf{[C2.]}] the active power injection from each DER $i$, $i \in \calN^g$, does not exceed its corresponding capacity limits, i.e., $\bm{\underline{p}}^g \leq~\bm{p}^g \leq \bm{\overline{p}}^g$;  
\item[\textbf{[C3.]}] the power flow on each line $\ell$, $\ell \in \calL$, does not exceed its maximum capacity, i.e., $-\overline{\bm{f}} \leq \bm{f} \leq \overline{\bm{f}}$.
\end{itemize}
Note that while constraint \textbf{C2} is a hard constraint that cannot be violated, constraint \textbf{C3} may be allowed to be violated slightly for a short period.
The ODCP can be formulated as the following optimization problem:
\begin{align*}
\minimize_{\bm{p}^g \in [\underline{\bm{p}}^g, \overline{\bm{p}}^g]} ~ c(\bm{p}^g),
\end{align*}
subject to 
\begin{subequations} \label{eq:ODCP}
\begin{gather}
	h(\bm{p}^g, \bm{p}^d, \bm{q}^d) = y^\star, \label{eq:ODCP_c1} \\
    -\overline{\bm{f}} \leq \bm{M}^{-1} (\bm{C} \bm{p}^g - \bm{p}^d) \leq \overline{\bm{f}},
\end{gather}
\end{subequations}
where $c(\cdot)$ denotes the cost function of the active power injections.
This problem is difficult, however, when the model describing the power exchange with the bulk power system, as captured by $h$, is unknown.
In this paper, we will resort to a data-driven approach to tackle this problem.

\section{DER Coordination Framework} \label{sec:framework}

In this section, we describe the building blocks of the proposed framework; namely an LTV IO model, an estimator, and a controller.

\subsection{Overview}

\begin{figure}[!t]
\centering
\includegraphics[width=3.5in]{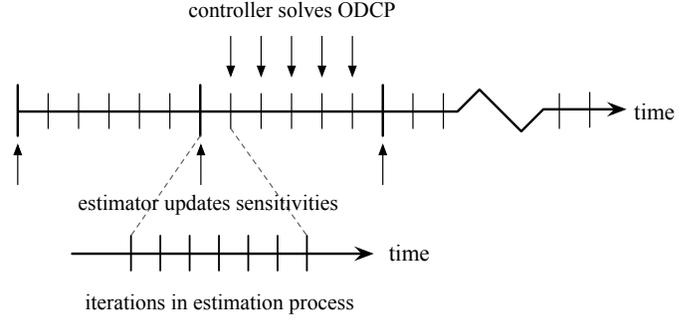}
\caption{Timescale separation of actions taken.}
\label{fig:timeline}
\end{figure}

The DER coordination framework consists of three components, namely (i) an LTC model of the system describing the relation between $y$ and $\bm{u}$, (ii) an estimator that provides estimates of the parameters---the so-called sensitivity vector---that populate the model in (i); and (iii)  a controller that uses the model in (i) with the  parameters estimated via (ii) to solve the ODCP.
This framework works on two timescales---a slow one and a fast one, as illustrated in Fig. \ref{fig:timeline}.
On the slow timescale, the controller determines the DER active power injection set-points by solving the ODCP periodically.
The sensitivity vector is also updated by the estimator periodically.
However, during each update of the sensitivity vector, the estimator needs to take actions in several iterations on a fast timescale.
Since the sensitivity vector may not change significantly in a short time, it is used in the ODCP for several time instants before it is updated again.
Before we proceed to presenting the detailed components in the proposed framework, we make the following assumption:
\begin{assumption} \label{a3}
	$\bm{p}^d$ and $\bm{q}^d$ are constant during the estimation process; therefore, changes in $y$ that occur across iterations in the estimation process depend only on changes in $\bm{p}^g$.
\end{assumption}
\begin{remark}
	Assumption \ref{a3} allows us to determine the impacts of the DER active power injections on the output. When the load variability is significant enough so that it cannot be neglected during the estimation process, it becomes necessary to measure the loads and determine their impacts on the output as well. This is beyond the scope of the paper, and therefore we will leave it as future work.
\end{remark}

\subsection{Input-Output System Model}

For notational simplicity in the later development, define $\bm{u} = \bm{p}^g$, $\underline{\bm{u}} = \underline{\bm{p}}^g$, $\overline{\bm{u}} = \overline{\bm{p}}^g$, and $\bm{\pi} = [(\bm{p}^d)^\top, (\bm{q}^d)^\top]^\top$; then, \eqref{eq:f} can be written as:
\begin{align} \label{eq:f2}
    y = h(\bm{u}, \bm{\pi}).
\end{align}

Unless otherwise noted, throughout this paper,  $x[k]$  denotes the value that some variable $x$ takes at iteration $k$.
It follows from \eqref{eq:f2} and Assumption \ref{a2} that $y[k-1] = h(\bm{u}[k-1], \bm{\pi})$ and $y[k] = h(\bm{u}[k], \bm{\pi})$.
Then, by the Mean Value Theorem, there exists $a_k \in [0, 1]$ and $\bm{\tilde{u}}[k] = a_k \bm{u}[k] + (1-a_k) \bm{u}[k-1]$ such that 
\begin{align*}
	y[k] - y[k-1] & = h(\bm{u}[k], \bm{\pi}) - h(\bm{u}[k-1], \bm{\pi}) \\
	& = \bm{\phi}[k]^\top (\bm{u}[k] - \bm{u}[k-1]),
\end{align*}
where $\bm{\phi}[k]^\top = [\phi_1[k], \cdots, \phi_n[k]] = \left.\dfrac{\partial h}{\partial \bm{u}}\right\vert_{\bm{\tilde{u}}[k]}$,\footnote{We adopt the convention that the partial derivative of a scalar function with respect to a vector is a row vector.} is referred to as the  sensitivity vector at iteration $k$.
It follows from Assumption~\ref{a3} that $\phi_i[k] \in [\underline{b}_1, \overline{b}_1]$, $i = 1, \cdots, n$.
Therefore, at any iteration $k$, \eqref{eq:f2} can be transformed into the following equivalent LTV IO model:
\begin{align} \label{eq:io_model}
	y[k] = y[k-1] + \bm{\phi}[k]^\top (\bm{u}[k] - \bm{u}[k-1]).
\end{align}

\subsection{Estimator On Fast Timescale}

As illustrated in Fig. \ref{fig:timeline}, the estimator updates the sensitivity vector across several iterations on the fast timescale.
At iteration $k$, the objective of the estimator is to obtain an estimate of $\bm{\phi}[k]$, denoted by $\bm{\hat{\phi}}[k]$,  using measurements collected up to iteration $k$, i.e., $y[k-1], \bm{u}[k-1], y[k-2], \bm{u}[k-2], \cdots$; we formulate this estimation problem as follows:
\begin{align*} 
\hat{\bm{\phi}}[k] = \argmin_{\hat{\bm{\phi}} \in \calB} ~  J^e(\hat{\bm{\phi}}) = \frac{1}{2} (y[k-1] - \hat{y}[k-1])^2,
\end{align*}
subject to 
\begin{align} \label{eq:estimator}
	\hat{y}[k-1] = y[k-2] + \hat{\bm{\phi}}^\top (\bm{u}[k-1] - \bm{u}[k-2]),
\end{align}
where $\calB = [\underline{b}_1, \overline{b}_1]^n$, $J^e(\cdot)$ is the cost function of the estimator, and $\hat{y}[k-1]$ is the value of $y[k-1]$ estimated by the IO model at iteration $k$.
Essentially, \eqref{eq:estimator} aims to find $\hat{\bm{\phi}}$ that minimizes the squared error between the estimated value and the true value of $y$.
Then, $\hat{\bm{\phi}}[k]$ is used in the controller to determine the control for the next time instant.

During the estimation process, it is still necessary to track the output target.
Therefore, at each iteration, the control is set based on the solution to the following problem:
\begin{align*} 
\bm{u}[k] = \argmin_{\bm{u} \in \mathcal{U}} ~  J^c(\bm{u}) = \frac{1}{2} (y^\star - \hat{y}[k])^2,
\end{align*}
subject to
\begin{align} \label{eq:controller}
	\hat{y}[k] = y[k-1] + \hat{\bm{\phi}}[k]^\top (\bm{u} - \bm{u}[k-1]),
\end{align}
where $\mathcal{U} = [\underline{\bm{u}}, \overline{\bm{u}}]$, and $J^c(\cdot)$ is the cost function.
Note that $\hat{\bm{\phi}}[k]$ is used in \eqref{eq:controller} to predict the value of $y[k]$ for a given $\bm{u}$.
Different from the ODCP, the objective of the controller during the estimation process is that the output tracks the target and there may exist multiple solutions to this problem.
This objective is chosen such that the DER active power injections behave in a way that can improve the estimation accuracy, as will be shown later in Section \ref{sec:algo}.

\subsection{Controller On Slow Timescale}

As illustrated in Fig. \ref{fig:timeline}, the controller solves the ODCP to determine the least-cost active power set-points for DERs on the slow timescale. Meanwhile, it also forces the DERs to inject random active power perturbations at each iteration on the fast timescale.
The ODCP to be solved by the controller is as follows:
\begin{align*} 
\minimize_{\bm{p}^g \in [\underline{\bm{p}}^g, \overline{\bm{p}}^g]} ~ c(\bm{p}^g),
\end{align*}
subject to 
\begin{subequations} \label{eq:ODCP2}
\begin{gather}
    y + \hat{\bm{\phi}}^\top (\bm{p}^g - \tilde{\bm{p}}^g) = y^\star, \\
	-\overline{\bm{f}} \leq \bm{M}^{-1} (\bm{C} \bm{p}^g - \bm{p}^d) \leq \overline{\bm{f}},
\end{gather}
\end{subequations}
where $y$ is the output and $\tilde{\bm{p}}^g$ is the DER active power injection vector at the beginning of the current time instant, and $\hat{\bm{\phi}}$ is the up-to-date sensitivity vector.
When $c$ is a convex function, \eqref{eq:ODCP2} is a convex problem and therefore can be solved easily with convergence guarantees.
Note that this formulation is obtained by replacing \eqref{eq:ODCP_c1} in \eqref{eq:ODCP} by the estimated IO model.

\section{Estimation Algorithm and Its Convergence} \label{sec:algo}

The ODCP can be solved using existing algorithms for convex optimization and thus we do not discuss it in more details here.
In this section, we focus on the problem faced by the estimator and propose a PGD based algorithm to solve it.
We then provide convergence results for the proposed algorithm.

\subsection{Estimation Algorithm}

We first describe the basic workflow of the proposed algorithm.
Each iteration consists of an estimation step and a control step.
At the beginning of iteration $k$, $y[k-1]$ is available to the estimator, which uses it to update the estimate of the sensitivity vector.
The updated estimate of the sensitivity vector, $\hat{\bm{\phi}}[k]$, is then used in the controller to determine the control, $\bm{u}[k]$.
Then, the DERs are instructed to change their active power injection set-points based on $\bm{u}[k]$.
At time instant $k+1$, the estimation and control iterations are repeated once $y[k]$ becomes available.
The sequential process described above, which happens on the fast timescale, is illustrated as follows:
\begin{align*}
	\cdots  \bm{u}[k-1] \rightarrow \underbrace{\bm{y}[k-1] \rightarrow \hat{\bm{\phi}}}_{\text{estimation step}} \overbrace{[k] \rightarrow \bm{u}[k]}^{\text{control step}} \rightarrow \bm{y}[k] \rightarrow \hat{\bm{\phi}}[k+1]   \cdots 
\end{align*}

Problems \eqref{eq:estimator} and \eqref{eq:controller} can be solved using the PGD method (see, e.g., \cite{iusem2003convergence}).
Let $\proj_{\mathbb{V}_1 \rightarrow \mathbb{V}_2}$ denote the projection operator from a vector space $\mathbb{V}_1$ to its (arbitrary) subspace $\mathbb{V}_2$, i.e.,
\begin{align*}
	\proj_{\mathbb{V}_1 \rightarrow \mathbb{V}_2}(\bm{v}_1) = \argmin_{\bm{v}_2 \in \mathbb{V}_2} \norm{\bm{v}_2 - \bm{v}_1},
\end{align*}
where $\bm{v}_1 \in \mathbb{V}_1$. 
For ease of notation, when the vector space to which $\bm{v}_1$ belongs is unambiguous, we simply write $\proj_{\mathbb{V}_2}(\bm{v}_1)$ instead of $\proj_{\mathbb{V}_1 \rightarrow \mathbb{V}_2}(\bm{v}_1)$.

Define the tracking error at iteration $k$ as $e[k] = y[k] - y^\star$.
In addition, define $\Delta y[k] = y[k] - y[k-1]$ and $\Delta \bm{u}[k] = \bm{u}[k] - \bm{u}[k-1]$.
The partial derivative vector of $J^e(\hat{\bm{\phi}})$ with respect to $\hat{\bm{\phi}}$ is
\begin{align} \label{eq:derv_e}
	\dfrac{\partial J^e(\hat{\bm{\phi}})}{\partial \hat{\bm{\phi}}} = \Delta \bm{u}[k-1] (\Delta \bm{u}[k-1]^\top \hat{\bm{\phi}} - \Delta y[k-1]),
\end{align}
and that of $J^c(\bm{u})$ with respect to $\bm{u}$ is
\begin{align} \label{eq:derv_c}
	\frac{\partial J^c(\bm{u})}{\partial \bm{u}} = \hat{\bm{\phi}}[k] (\hat{\bm{\phi}}[k]^\top (\bm{u} - \bm{u}[k-1]) + e[k-1]).
\end{align}

Instead of solving both \eqref{eq:estimator} and \eqref{eq:controller} to completion, we iterate the PGD algorithm that would solve them for one step at each iteration.
Specifically, at iteration $k$, we evaluate the new gradient at $\hat{\bm{\phi}}[k-1]$  and $\bm{u}[k-1]$ and iterate once.
Thus, by using \eqref{eq:derv_e} and \eqref{eq:derv_c}, the update rules for $\hat{\bm{\phi}}$ and $\bm{u}$, respectively, are
\begin{align} \label{eq:estimation}
		\hat{\bm{\phi}}[k] = & \proj_{\calB}\left( \hat{\bm{\phi}}[k-1] - \alpha_k \Delta \bm{u}[k-1] \right. \nonumber \\
		 & \left. (\Delta \bm{u}[k-1]^\top \hat{\bm{\phi}}[k-1] - \Delta y[k-1] ) \right),
\end{align}
\begin{align} \label{eq:control0}
	\bm{u}[k] = \proj_{\mathcal{U}}\left( \bm{u}[k-1] - \beta_k e[k-1] \hat{\bm{\phi}}[k] \right),
\end{align}
where $\alpha_k > 0$ and $\beta_k > 0$ are the estimation and control step sizes at iteration $k$.

\begin{algorithm}[!t]
    \SetAlgoLined
    \DontPrintSemicolon
    \KwData{\\~~~~~~$y$: output \\~~~~~~$y^\star$: output tracking target \\~~~~~~$\delta$: maximum allowed tracking error \\~~~~~~$\hat{\bm{\phi}}^0$: initial estimate of sensitivity vector \\~~~~~~$\bm{u}^0$: initial DER active power injection set-point}
    \KwResult{\\~~~~~~$\bm{u}$: DER active power injection set-point \\~~~~~~$\hat{\bm{\phi}}$: estimate of sensitivity vector}
    \textbf{Initialization}: set $\hat{\bm{\phi}}[0] = \hat{\bm{\phi}}^0$, $\bm{u}[-1] = \bm{u}[0] = \bm{u}^0$, obtain measurement of $y[-1]$, set $k = 1$\;
    \While{$|e[k]| > \delta$}{
        obtain new measurement of $y[k-1]$\;
        compute $e[k-1] = y[k-1] - y^\star$\;
        compute $\Delta y[k-1] = y[k-1] - y[k-2]$\;
        compute $\Delta \bm{u}[k-1] = \bm{u}[k-1] - \bm{u}[k-2]$\;
        update the sensitivity vector estimate according to
        \begin{align*}
				\hat{\bm{\phi}}[k] = & \proj_{\calB}\left( \hat{\bm{\phi}}[k-1] - \alpha_k \Delta \bm{u}[k-1] \right. \\
				& \left. (\Delta \bm{u}[k-1]^\top \hat{\bm{\phi}}[k-1] - \Delta y[k-1] ) \right)
		\end{align*}\; \vspace{-0.2in}
		update the control vector, $\bm{u}$, according to
		\begin{align*}
			\bm{u}[k] = \proj_{\mathcal{U}}\left( \bm{u}[k-1] - \beta_k e[k-1] \bm{W}[k] \hat{\bm{\phi}}[k] \right)
		\end{align*}\; \vspace{-0.2in}
		change DER active power injections to $\bm{u}[k]$\;
		set  $k = k + 1$\;
    }
\label{algo:DER_coordination}
\caption{Estimation Algorithm}
\end{algorithm}

In order to resolve the potential issue of collinearity in the measurements used by the estimator, we introduce random perturbations during the estimation process.
Define $\bm{W}[k] = \diag{w_1[k], \dots, w_n[k]}$, where $w_i[k]$'s are independent random variables that follow a Bernoulli distribution with a probability parameter of $0.5$.
Then, the control update rule in \eqref{eq:control0} is modified, resulting in:
\begin{align} \label{eq:control}
	\bm{u}[k] = \proj_{\mathcal{U}}\left( \bm{u}[k-1] - \beta_k e[k-1] \bm{W}[k] \hat{\bm{\phi}}[k] \right).
\end{align}
Intuitively, this means that at each iteration, the control of each DER is updated with a probability of $0.5$. 
The random perturbation in the control is key to establish convergence of the parameter estimation process.
The estimation algorithm, along with its initialization, is summarized in Algorithm \ref{algo:DER_coordination}, where $\bm{u}^0$ is the vector of DER active power injections at the beginning of the time instant at which the estimation starts and $\hat{\bm{\phi}}^0$ is the up-to-date estimate of the sensitivity vector at the beginning of the same time instant.

\subsection{Convergence Analysis}

The main convergence result for the control step during the estimation process is stated next.
\begin{theorem} \label{thm1}
	Using the estimation update rule in \eqref{eq:estimation} and the  control update rule in \eqref{eq:control} with $\beta_k \in (\frac{\epsilon}{\underline{b}_1^2}, \frac{1}{n \overline{b}_1^2})$, where $0 < \epsilon < \frac{\underline{b}_1^2}{n \overline{b}_1^2}$ is a given parameter, the system attains one of the following equilibria: 1) $e[k]$ converges to $0$ a.s.; 2) $e[k]$ converges to some positive constant and $\bm{u}[k]$ stays at $\underline{\bm{u}}$; 3) $e[k]$ converges to some negative constant and $\bm{u}[k]$ stays at $\overline{\bm{u}}$. In all cases, $\lim_{k \to \infty} \Delta \bm{u}[k] = \zeros_n$, where $\zeros_n \in \bbR^n$ is an all-zeros vector.
\end{theorem}

Theorem \ref{thm1} shows something intuitive, i.e.,  the tracking error will be positive (negative) if the   requested active power is less (more) than the minimum (maximum) amount of active power the DERs can provide; otherwise, the tracking error goes to zero a.s.

We note that $\epsilon$ has a direct impact on the convergence rate of the control algorithm. This is more obvious in a deterministic setting, when the control update rule in \eqref{eq:control0} is used instead of the one in \eqref{eq:control}.
A result on the convergence rate is given in the following corollary.
\begin{corollary} \label{corollary:1}
	Assume $\bm{u}[k] \neq \underline{\bm{u}}$ and $\bm{u}[k] \neq \overline{\bm{u}}$, $\forall k \in \bbN$. 
	Using the estimation update rule in \eqref{eq:estimation} and  the control update rule in \eqref{eq:control0} with $\beta_k \in (\frac{\epsilon}{\underline{b}_1^2}, \frac{1}{n \overline{b}_1^2})$, where $\epsilon > 0$ is a given parameter, $e[k]$ converges to $0$ at a rate smaller that $1-\epsilon$, i.e., $\left|\frac{e[k]}{e[k-1]}\right| < 1- \epsilon$.
\end{corollary}

We refer the readers to Appendix \ref{appendix:thm1} for detailed proofs of the convergence results above.

Next, we establish  the convergence of the estimation step.
Define the estimation error vector at iteration $k$ as $\bm{\varepsilon}[k] = \hat{\bm{\phi}}[k] - \bm{\phi}[k]$. 
Since both $\hat{\bm{\phi}}[k]$ and $\bm{\phi}[k]$ are bounded, $\bm{\varepsilon}[k]$ is also bounded.
Define $\Delta \bm{\phi}[k] = \bm{\phi}[k] - \bm{\phi}[k-1]$.
The convergence result for the estimation step is stated next.
\begin{theorem} \label{thm2}
	Using the estimation update  rule in \eqref{eq:estimation} and the  control update rule \eqref{eq:control}, with $\alpha_{k+1} = \frac{2}{\norm{\Delta \bm{u}[k]}^2}$, $\beta_k \in (\frac{\epsilon}{n \underline{b}_1^2}, \frac{1}{n \overline{b}_1^2})$, where $0 < \epsilon < \frac{\underline{b}_1^2}{\overline{b}_1^2}$ is a given parameter, if $\bm{u}[k] \in (\underline{\bm{u}}, \overline{\bm{u}})$ and $e[k] \neq 0$, $\forall k \in \bbN$, then $\norm{\bm{\varepsilon}[k]}$ converges to $0$ a.s.
\end{theorem}
The intuition is that the estimation error goes to zero if the system can be continuously excited (guaranteed by the condition $\bm{u}[k] \in (\underline{\bm{u}}, \overline{\bm{u}})$ and $e[k] \neq 0$, $\forall k \in \bbN$).

We refer the readers to Appendix \ref{appendix:thm2} for detailed proofs of the convergence result above.

\section{Numerical Simulations} \label{sec:simu}

\begin{figure}[!t]
\centering
\includegraphics[width=3.5in]{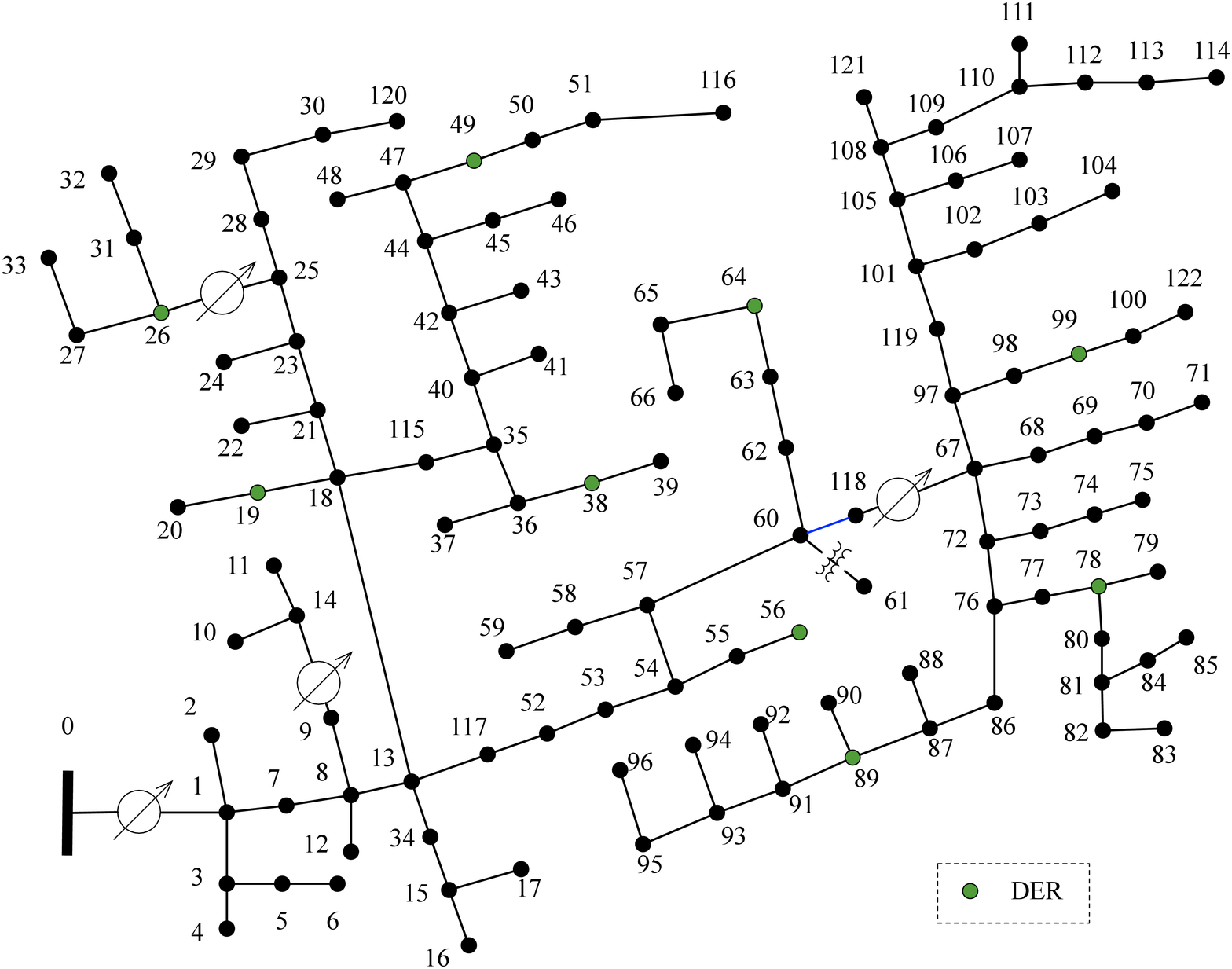}
\caption{IEEE 123-bus distribution test feeder.}
\label{fig:123bus}
\end{figure}

In this section, we illustrate the application of the proposed DER coordination framework and validate the theoretical results presented earlier. 
From a practical point of view, the timescale separation illustrated in Fig. \ref{fig:timeline} is critical for the applicability of the proposed framework.
Specifically, the estimation process needs to be much faster than the timescale governing the load changes.
The DERs, which are typically power electronics interfaced, can respond very quickly, on a timescale of millisecond to second \cite{ajala2017hierarchy}.
In this simulation, we set the duration between two iterations to be $100$~millisecond.
We will show later that under this setup, the requirements on the time separation can be reasonably met.

A modified three-phase balanced IEEE 123-bus distribution test feeder from \cite{test_feeder} (see Fig. \ref{fig:123bus} for the one-line diagram) is used for all numerical simulations.
This balanced test feeder has a total active power load of $3000$~kW, and a total reactive power load of $1575$~kVAr.
DERs are added at buses 19, 26, 38, 49, 56, 64, 78, 89, 99.
We assume each DER can output active power from $0$~kW to $100$~kW.
Therefore, the maximum DER active power injections accounts for $30$\% of the nominal loads.
To illustrate the impacts of reactive power control, we assume all DERs operate at unity power factor except DERs at buses 78 and 89, which are assumed to have enough reactive power capacity and maintain a constant voltage magnitude of $0.95$ p.u.
Yet, we would like to emphasize that the proposed algorithm is agnostic to the underlying reactive power control scheme and also works under other reactive power control schemes.
In addition, to validate the effectiveness of the proposed algorithm under different operating conditions of the power distribution system, we assume there are some uncontrollable renewable energy resources in the power distribution system, which are modeled as negative loads.
The underlying nonlinear power flow problem is solved using Matpower \cite{matpower}.

In all subsequent simulations, we set $\underline{b}_1=0.8$, $\overline{b}_1=1.2$, which are reasonable values for real power systems.
Intuitively, these values indicate that the percentage of active power losses will be no larger than $20\%$ of the total active power injections.
Note that the exact value of $b_2$ is not necessary.
Under this simulation setting, as given in Theorem \ref{thm1}, the upper bound of the control step size is $\frac{1}{n\overline{b}_1^2} \approx 0.0694$.

We note that comprehensive simulations including the two timescales can be done using data such as ones adopted in \cite{robbins2016optimal}.
However, since the ODCP to be solved on the slow timescale is a standard problem, we will mainly focus on simulations for the fast timescale, where our major contributions lie.

\subsection{Case I}

In this case, we assume the power distribution system is importing energy from the bulk power system with $y = -3110$~kW.
This corresponds to the situation where the uncontrollable renewable energy resources are not generating more active power than that needed by the loads.
In addition, we set $\hat{\bm{\phi}}^0 = \ones_n$ and $\bm{u}^0 = \zeros_n$, where $\ones_n \in \bbR^n$ is an all-ones vector.

\subsubsection{Tracking Performance During Estimation}

\begin{figure}[!t]
\centering
\includegraphics[width=3.5in]{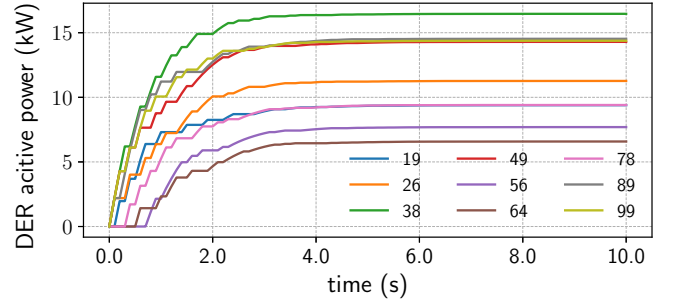}
\caption{DER active power injections for $\beta_k = 0.02$ and $y^\star=-3000$ ~kW in Case I. (Legends indicate DER buses.)}
\label{fig:control}
\end{figure}

\begin{figure}[!t]
\centering
\includegraphics[width=3.5in]{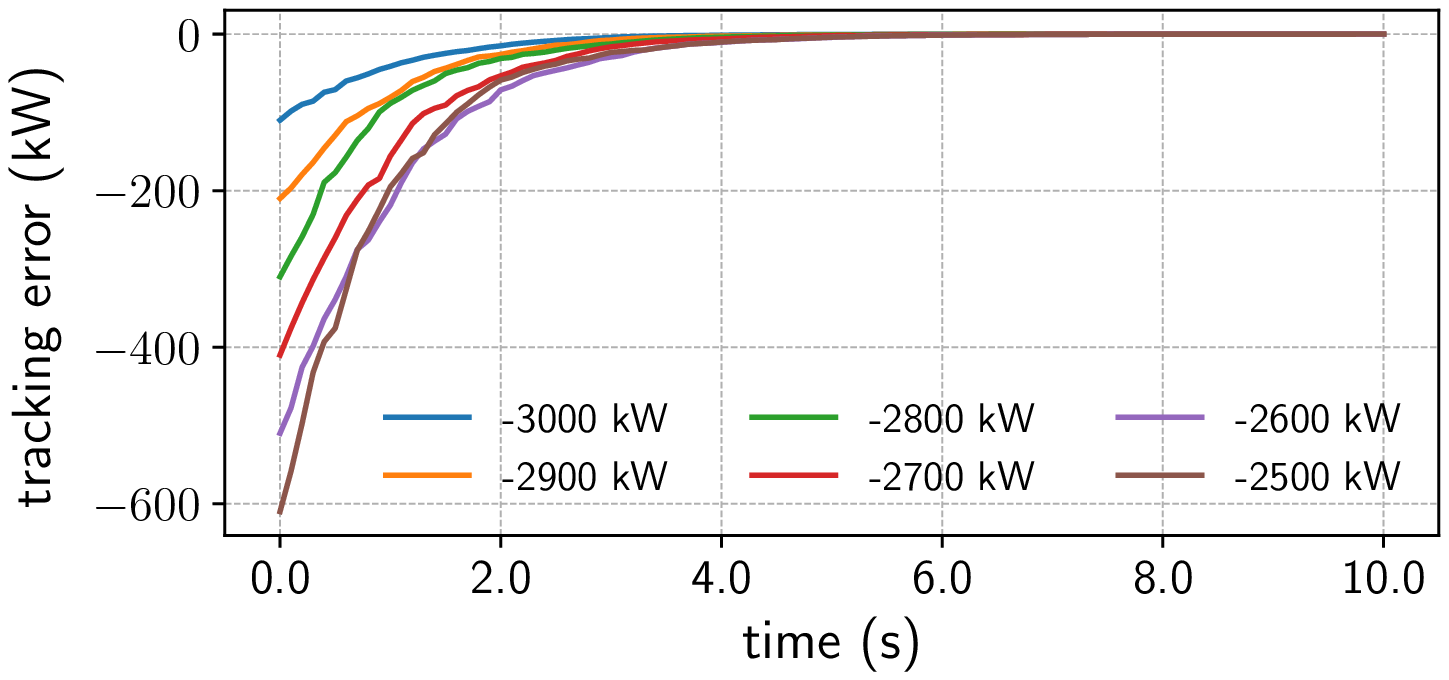}
\caption{Tracking error for $\beta_k = 0.02$ under various tracking targets in Case I. (Legends indicate values of $y^\star$.)}
\label{fig:tracking_error_output_target_import}
\end{figure}

\begin{figure}[!t]
\centering
\includegraphics[width=3.5in]{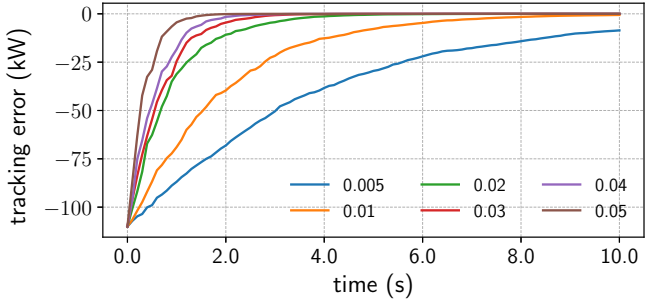}
\caption{Tracking error for $y^\star=-3000$~kW and various constant control step sizes in Case I. (Legends indicate values of $\beta_k$.)}
\label{fig:tracking}
\end{figure}

For $y^\star = -3000$~kW and a constant step size $\beta_k = 0.02$, the DER active power injections are shown in Fig.~\ref{fig:control}.
The non-smoothness in the active power profiles is caused by the random perturbation imposed on the control step.
Also as shown in Fig. \ref{fig:tracking_error_output_target_import}, the convergence rate of the tracking error is not affected by the tracking target, i.e., the total active power required from the bulk power system. The tracking error $e[k]$ under various constant control step sizes is shown in Fig.~\ref{fig:tracking}.
As expected, a larger step size will reduce the tracking error faster than a small step size.

\subsubsection{Estimation Accuracy}

\begin{table}[!t]
\caption{Estimated Sensitivities in Case I After $60$ Iterations}
\label{table:estimation_error_1}
\centering
\begin{tabular}{ccccccc}
\toprule
bus & 19 & 26 & 38 & 49 & 56  \\
\midrule
true & 1.0394 & 1.0413 & 1.0426 & 1.0454 & 1.0467 \\
estimate & 1.0342 & 1.0390 & 1.0440 & 1.0468 & 1.0421 \\
\midrule
\midrule
bus & 64 & 78 & 89 & 99  \\
\midrule
true & 1.0702 & 1.0703 & 1.0749 & 1.0711 \\
estimate & 1.0696 & 1.0697 & 1.0817 & 1.0702 \\
\bottomrule
\end{tabular}
\end{table}

\begin{figure}[!t]
\centering
\includegraphics[width=3.5in]{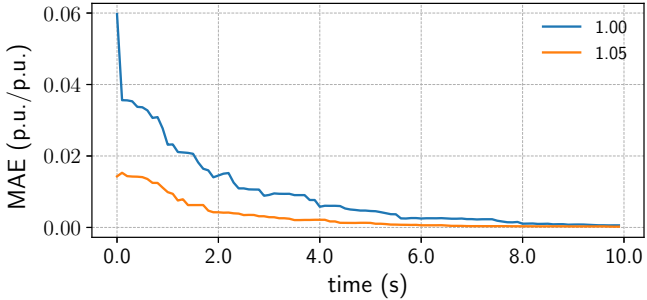}
\caption{MAE of estimation errors with $\beta_k = 0.02$ in Case I. (Legends indicate the values of initial estimates.)}
\label{fig:estimate_error}
\end{figure}

With $\beta_k=0.02$ and $y^\star=-3000$~kW, true and estimated sensitivities are compared in Table \ref{table:estimation_error_1} and the mean absolute error (MAE) of estimation errors, i.e., $\frac{\sum_{i=1}^n |\varepsilon_i[k]|}{N}$, is plotted in Fig. \ref{fig:estimate_error}.
The estimated sensitivities are very close to their true values after $60$ steps, which corresponds to $6$~s.
Note that $\hat{\bm{\phi}}^0$ has an important impact on the convergence of the sensitivity estimation algorithm.
As can be seen from Fig. \ref{fig:estimate_error}, when the initial values of the estimated sensitivities are set to $1.05$, which is closer to their true values, it takes much less time to obtain a small estimation error.

\begin{figure}[!t]
\centering
\includegraphics[width=3.5in]{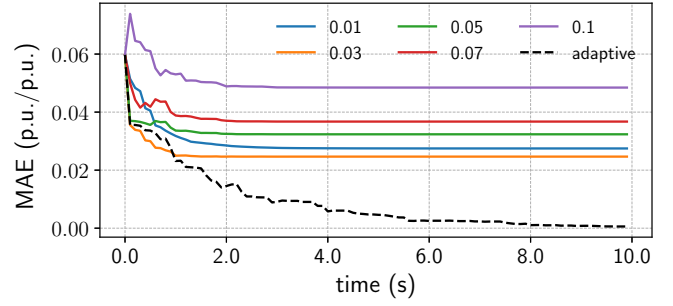}
\caption{MAE of estimation errors with $\beta_k=0.02$ under various estimation step sizes in Case I. (Legends indicate values of $\alpha_k$.)}
\label{fig:estimation_error_step_size_e}
\end{figure}

\begin{figure}[!t]
\centering
\includegraphics[width=3.5in]{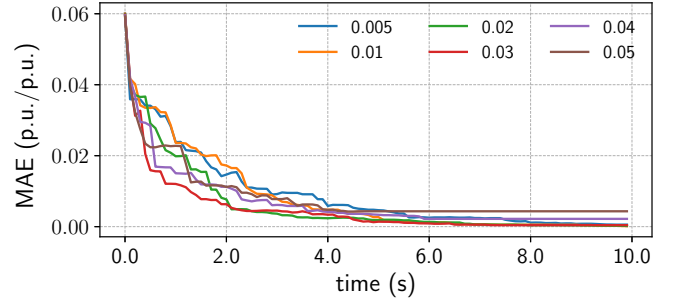}
\caption{MAE of estimation errors under various control step sizes in Case I. (Legends indicate values of $\beta_k$.)}
\label{fig:estimation_error_step_size_c}
\end{figure}

While the estimation step size, $\alpha_k$, in the proposed algorithm is adaptive, we also investigate the case when $\alpha_k$ is chosen to be constant.
Figure \ref{fig:estimation_error_step_size_e} shows the MAE of estimation error under various constant estimation step sizes.
As can be seen from Fig. \ref{fig:estimation_error_step_size_e}, the MAE of estimation will converge to some non-zero constant under constant estimation step sizes.

The impact of the control step sizes on the estimation accuracy is also investigated.
Figure~\ref{fig:estimation_error_step_size_c} shows the MAE of estimation errors under various control step sizes.
With a large control step size, the tracking error converges to $0$ quickly, leading to a situation where the system cannot get sufficient excitation and consequently, the estimation errors cannot be further reduced.

\subsection{Case II}

In this case, we assume the power distribution system is exporting energy to the bulk power system with $y = 1000$~kW.
This corresponds to the situation where the uncontrollable renewable energy resources are generating much more active power than that needed by the loads.
We set $\hat{\bm{\phi}}^0 = \ones_n$ and $\bm{u}^0 = \zeros_n$.

\subsubsection{Tracking Performance During Estimation}

\begin{figure}[!t]
\centering
\includegraphics[width=3.5in]{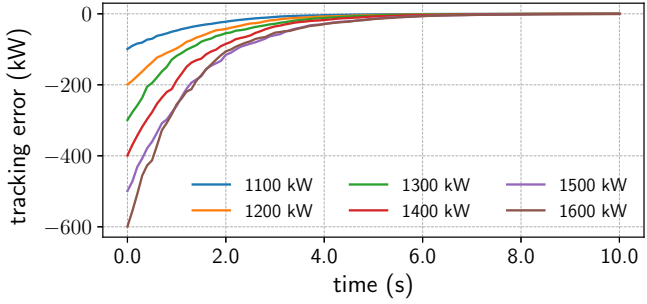}
\caption{Tracking error for $\beta_k = 0.02$ under various tracking targets in Case II. (Legends indicate values of $y^\star$.)}
\label{fig:tracking_error_output_target_export}
\end{figure}

Using a constant step size $\beta_k = 0.02$, the convergence rate of the tracking error under various tracking target is shown in Fig. \ref{fig:tracking_error_output_target_export}.
Similar to results in Case I, the convergence rate is not affected by the tracking target.

\subsubsection{Estimation Accuracy}

\begin{table}[!t]
\caption{Estimated Sensitivities in Case II After $60$ Iterations}
\label{table:estimation_error_export}
\centering
\begin{tabular}{ccccccc}
\toprule
bus & 19 & 26 & 38 & 49 & 56  \\
\midrule
true & 0.9533 & 0.9526 & 0.9518 & 0.9509 & 0.9254 \\
estimate & 0.9588 & 0.9512 & 0.9497 & 0.9475 & 0.9285 \\
\midrule
\midrule
bus & 64 & 78 & 89 & 99  \\
\midrule
true & 0.8872 & 0.8477 & 0.8488 & 0.8700 \\
estimate & 0.8854 & 0.8536 & 0.8396 & 0.8707 \\
\bottomrule
\end{tabular}
\end{table}

\begin{figure}[!t]
\centering
\includegraphics[width=3.5in]{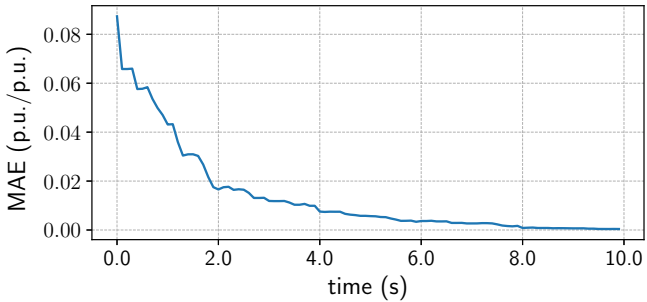}
\caption{MAE of estimation errors with $\beta_k = 0.02$ in Case II.}
\label{fig:estimation_error_export}
\end{figure}

With $\beta_k=0.02$ and $y^\star=1100$~kW, true and estimated sensitivities are compared in Table \ref{table:estimation_error_export} and the MAE of estimation errors is plotted in Fig. \ref{fig:estimation_error_export}, respectively.
Similar to the results in Case I, the estimated sensitivities are very close to their true values after $60$ steps.
This verifies that the proposed estimation algorithm can effectively estimate the sensitivities under different operating conditions of the power distribution system.

We note that the performance of the proposed estimation algorithm is independent of the number of DERs.
To see this, we simulate a case where the DER at bus 99 gets disconnected and consequently, there are $8$ DERs in the power distribution system.
The sensitivities at these $8$ DER buses can still be estimated effectively, as is shown in Fig. \ref{fig:estimation_error_DER_loss}.

\begin{figure}[!t]
\centering
\includegraphics[width=3.5in]{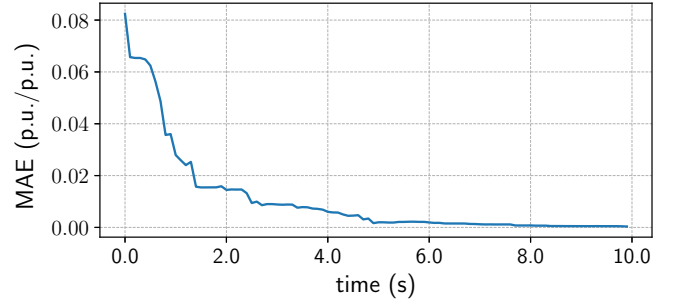}
\caption{MAE of estimation errors with $\beta_k = 0.02$ in Case II with $8$ DERs.}
\label{fig:estimation_error_DER_loss}
\end{figure}

\subsection{Case III}

In this case, we illustrate how the proposed framework handles line congestions.
The setup is the same as Case II except that the tracking target is $y^\star=1500$~kW.
We set the capacity limit of line $(55, 56)$ to $40$~kW to create congestion.
For simplicity, all other lines are assumed to have an infinite capacity; yet, we would like to emphasize that the proposed framework can handle multiple line congestions.
The objective function in the ODCP is assumed to be $c(\bm{p}^g) = \norm{\bm{p}^g - \tilde{\bm{p}}^g}^2$, where $\tilde{\bm{p}}^g$ is the current DER active power injection vector as used in \eqref{eq:ODCP2}.
Intuitively, this objective function will favor the solution with the least change in the DER active power injections.

The estimation algorithm is first ran to obtain an estimate of the sensitivity vector.
After the estimation algorithm ends, the DER at bus 56 is generating $51.3$~kW, which exceeds the capacity limit of line $(55, 56)$.
The ODCP is run afterwards to adjust the active power set-points of the DERs.
Figure \ref{fig:set_point_ODCP} shows the DER active power set-point before and after solving the ODCP.
The DER at bus 56 is dispatched down to $40$~kW, which conforms with the capacity limit of line $(55, 56)$.
All other DERs are dispatched up such that the active power exchanged between the distribution and bulk power systems still equals to $1500$~kW.
We note that line $(55, 56)$ is overloaded for a short period during the estimation process but the is quickly restored to a normal loading level after the DER active power set-points are adjusted via the ODCP.

\begin{figure}[!t]
\centering
\includegraphics[width=3.5in]{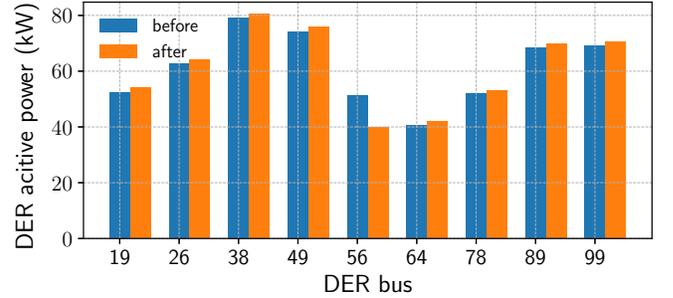}
\caption{DER active power set-point before and after solving the ODCP in Case III.}
\label{fig:set_point_ODCP}
\end{figure}

\section{Concluding Remarks} \label{sec:con}

In this paper, we proposed a data-driven coordination framework for DERs in a lossy power distribution system, the model of which is not completely known, to collectively provide some pre-specified amount of active power to a bulk power system at a minimum generating cost, while respecting distribution line capacity limits.
The proposed framework consists of a LTV IO model, an estimator, and a controller.

We showed that using the estimation algorithm proposed in the framework, the estimated parameters converge to the true parameters a.s., and the total provided active power converges to the required amount during the estimation process.
The data-driven nature of this framework makes it adaptive to various system operating conditions.
We validated the effectiveness of the proposed framework through numerical simulations on a modified version of the IEEE 123-bus test feeder.

There are two potential directions for future work.
The first one is to develop efficient estimation algorithms for scenarios in which the estimation process (with random perturbations) cannot be executed on a timescale that is much faster than the one on which the loads vary, and quantifies the impacts of load variability on the performance of the algorithm; this essentially relaxes Assumption \ref{a3}.
The second one is to extend the proposed estimation algorithms to scenarios where there are multiple connection buses between the distribution system and the bulk power system, and consequently, multiple outputs.
\color{black}

\appendix \label{appendix}
\numberwithin{equation}{section}
\setcounter{equation}{0}

\subsection{Proof of Theorem \ref{thm1}} \label{appendix:thm1}

The convergence analysis of the control step during the estimation process relies on the following two lemmas.
\begin{lemma} \label{lemma:remove_proj}
	There exists $\bar{\bm{\phi}}[k]$ satisfying $\zeros_n \leq \bar{\bm{\phi}}[k] \leq \hat{\bm{\phi}}[k]$, such that \eqref{eq:control} is equivalent to
	\begin{align*}
		\bm{u}[k] = \bm{u}[k-1] - \beta_k e[k-1] \bm{W}[k] \bar{\bm{\phi}}[k].
	\end{align*}
	Also, $\bar{\bm{\phi}}[k] = \zeros_n$ if and only if $\bm{u}[k] = \underline{\bm{u}}$ or $\bm{u}[k] = \overline{\bm{u}}$.
	Furthermore, if $\bm{u}[k] \neq \underline{\bm{u}}$ and $\bm{u}[k] \neq \overline{\bm{u}}$, there exists $i \in \calN^g$ such that $\bar{\phi}_i[k] = \hat{\phi}_i[k] \in [\underline{b}_1, \overline{b}_1]$.
\end{lemma}
\begin{proof} If $\bm{u}[k-1] - \beta_k e[k-1] \bm{W}[k] \hat{\bm{\phi}}[k] \in \mathcal{U}$, then we simply set $\bar{\bm{\phi}}[k] = \hat{\bm{\phi}}[k]$.
	Without loss of generality, first consider the case where the following holds for some $i \in \calN^g$:
	\begin{align} \label{eq:l2_1}
		u_i[k-1] - \beta_k e[k-1] w_i[k] \hat{\phi}_i[k] > \overline{u}_i.
	\end{align}
	Then, $e[k-1] < 0$ and $w_i[k] > 0$ since otherwise \eqref{eq:l2_1} cannot not hold.
	Therefore,
	\begin{align} \label{eq:l2_3}
		u_i[k] = \proj_{\mathcal{U}}(u_i[k-1] - \beta_k e[k-1] w_i[k] \hat{\phi}_i[k]) = \overline{u}_i.
	\end{align}
	Let $\bar{\phi}_i[k] = \frac{u_i[k-1] - \overline{u}_i}{\beta_k e[k-1] w_i[k]}$; by definition, $\bar{\phi}_i[k] = 0$ if and only if $u_i[k-1] = \overline{u}_i$.
	Then, we have that:
	\begin{align} \label{eq:l2_4}
		u_i[k] = u_i[k-1] - \beta_k e[k-1] w_i[k] \bar{\phi}_i[k].
	\end{align}
	If follows from \eqref{eq:l2_1}, \eqref{eq:l2_3}, and \eqref{eq:l2_4} that
	\begin{align} \label{eq:l2_5}
		\beta_k e[k-1] \hat{\phi}_i[k] w_i[k] < \beta_k e[k-1] \bar{\phi}_i[k] w_i[k],
	\end{align}
	which leads to $0 \leq \bar{\phi}_i[k] < \hat{\phi}_i[k]$.
	A similar argument can be used to for the case where $u_i[k-1] - \beta_k e[k-1] w_i[k] \hat{\phi}_i[k] < \underline{u}_i$ and for some $i \in \calN^g$.
	
	If $\bm{u}[k] \neq \underline{\bm{u}}$ and $\bm{u}[k] \neq \overline{\bm{u}}$, then there exists $i \in \calN^g$ such that $\underline{u}_i < u_i[k] < \overline{u}_i$, which implies
	\begin{align}
		u_i[k] = u_i[k-1] - \beta_k e[k-1] w_i[k] \hat{\phi}_i[k].
	\end{align}
	Therefore, $\bar{\phi}_i[k] = \hat{\phi}_i[k]$. 
	Consequently, $\bar{\phi}_i[k] = \hat{\phi}_i[k] \in [\underline{b}_1, \overline{b}_1]$.
	It can be easily seen that if $\mathcal{U}$ is sufficiently large and no DER hits its capacity limits, then  $\bar{\bm{\phi}}[k] = \hat{\bm{\phi}}[k]$.
\end{proof}

\begin{lemma} \label{lemma:rv_product}
	Let $X_k$, $k=1, 2, \cdots$, be independently identically distributed (i.i.d.) random variables. Assume $X_k > 0$ and $\expect{X_k} \in (0, 1)$, where $\mathbb{E}$ denotes expectation. Let $Y_k = \prod_{i=1}^k X_i$. Then, $\lim_{k \to \infty} Y_k = 0$ a.s.
\end{lemma}
\begin{proof}
	Note that $Y_k = \exp{\sum_{i=1}^k \log X_i}$. By the Strong Law of Large Numbers (see Proposition 2.15 in \cite{hajek2015random}), we have that
	\begin{align}
		\lim_{k \to \infty} \sum_{i=1}^k \frac{1}{k} \log X_i = \expect{\log X_1},~\text{a.s.}
	\end{align}
	By Jensen's inequality (see Theorem 2.18 in \cite{hajek2015random}), we have that
	\begin{align}
		\expect{\log X_1} \leq \log \expect{X_1} < 0.
	\end{align}
	Therefore,
	\begin{align}
		\lim_{k \to \infty} \sum_{i=1}^k k \frac{1}{k} \log X_i = -\infty,~\text{a.s.},
	\end{align}
	which leads to
	\begin{align}
		\lim_{k \to \infty} Y_k &=  \lim_{k \to \infty} \exp{\sum_{i=1}^k k \frac{1}{k} \log X_i} \nonumber \\
		&= \exp{\lim_{k \to \infty} \sum_{i=1}^k k \frac{1}{k} \log X_i} \nonumber \\
		&= 0,~\text{a.s.};
	\end{align}
	this completes the proof.
\end{proof}

Using Lemma \ref{lemma:remove_proj} and Lemma \ref{lemma:rv_product}, we can prove the Theorem \ref{thm1} as follows:
\begin{proof}
	By \eqref{eq:io_model}, we have that
	\begin{align} \label{thm1:1}
		e[k] - e[k-1] = \bm{\phi}[k]^\top \Delta \bm{u}[k].
	\end{align}
	By Lemma \ref{lemma:remove_proj}, we have that
	\begin{align} \label{thm1:2}
		\Delta \bm{u}[k] = - \beta_k e[k-1] \bm{W}[k] \bar{\bm{\phi}}[k],
	\end{align}
	where $\zeros_n \leq \bar{\bm{\phi}}[k] \leq \hat{\bm{\phi}}[k]$.
	Substituting \eqref{thm1:2} into \eqref{thm1:1} leads to
	\begin{align} \label{thm1:error1}
		e[k] = (1 - \beta_k \bm{\phi}[k]^\top \bm{W}[k] \bar{\bm{\phi}}[k]) e[k-1].
	\end{align}
	Define $\rho_k = 1 - \beta_k \bm{\phi}[k]^\top \bm{W}[k] \bar{\bm{\phi}}[k]$, then
	\begin{align} \label{thm1:error2}
		e[k] = e[0] \prod_{i=1}^{k} \rho_i.
	\end{align}
	By Assumption \ref{a3}, $0 < \underline{b}_1 \leq \phi_i[k] \leq \overline{b}_1$. 
	In addition, it follows from Lemma \ref{lemma:remove_proj} that $0 \leq \bar{\phi}_i[k] \leq \hat{\phi}_i[k] \leq \overline{b}_1$.
	Therefore, $\bm{\phi}[k]^\top \bm{W}[k] \bar{\bm{\phi}}[k]$ can be bounded as follows:
	\begin{align} \label{eq:upper_bouned}
		0 \leq \bm{\phi}[k]^\top \bm{W}[k] \bar{\bm{\phi}}[k] = \sum_{i=1}^n w_i[k] \phi_i[k] \bar{\phi}_i[k] \leq n \overline{b}_1^2.
	\end{align}
	Since $\beta_k < \frac{1}{n \overline{b}_1^2}$, then all $e[k]$ has the same sign for all $k$ (positive if $e[0]>0$, and negative otherwise). As a result, the entries of $\Delta \bm{u}[k]$ always have the same sign by \eqref{thm1:2}.
	
	\noindent (\textit{a}) If $e[k] = 0$ for some $k \in \bbN$, then $\Delta \bm{u}[k+1] = \zeros_n$.
	The control and estimation algorithms will stop updating according to \eqref{eq:estimation} and \eqref{eq:control}. 
	In this case, $\bm{u}[k]$ may equal to $\underline{\bm{u}}$ or $\overline{\bm{u}}$ or neither, and
	the system attains an equilibrium.
	
	\noindent (\textit{b}) Now suppose $e[k] \neq 0, \forall k \in \bbN$. Since the increments of $\bm{u}$ always have the same sign, the entries of $\bm{u}$ cannot hit their bounds in different directions, i.e., some hit their lower bounds while others hit their upper bounds.
	
	\noindent (\textit{b}.1) If $\bm{u}[k] = \underline{\bm{u}}$ for some iteration $k$, then $e[k] > 0, \forall k \in \bbN$.
	By \eqref{eq:control}, we have that
	\begin{align}
		\bm{u}[k+1] = \proj_{\mathcal{U}} ( \underline{\bm{u}} - \beta_{k} e[k] \bm{W}[k+1] \hat{\bm{\phi}}[k+1]) = \underline{\bm{u}}.
	\end{align}
	Thus, $\Delta \bm{u}[k+1] = \zeros_n$, which leads to $e[k+1] = e[k]$ by \eqref{thm1:1}.
	Therefore, $\bm{u}$ will equal to $\underline{\bm{u}}$ and $e[k']=e[k] > 0$ for all $k'>k$.
	
	Similarly, when $\bm{u}[k] = \overline{\bm{u}}$, $\bm{u}$ will be equal to $\overline{\bm{u}}$, and $e$ will be equal to $e[k] < 0$ in all future time intervals.
	The system attains an equilibrium in both cases.
	
	\noindent (\textit{b}.2) If $\bm{u}[k] \neq \underline{\bm{u}}$ and $\bm{u}[k] \neq \overline{\bm{u}}$, $\forall k \in \bbN$, by Lemma \ref{lemma:remove_proj}, there exists $i \in \calN^g$ such that $\bar{\phi}_i[k] \in [\underline{b}_1, \overline{b}_1]$. 
	Then,
	\begin{align} \label{eq:lower_bouned}
		\bm{\phi}[k]^\top \bm{W}[k] \bar{\bm{\phi}}[k] = \sum_{i=1}^n \phi_i [k] \bar{\phi}_i[k] w_i[k] \geq \underline{b}_1^2 w_i[k].
	\end{align}
	Thus, by using \eqref{eq:upper_bouned} and \eqref{eq:lower_bouned}, it follows that $\rho_k \in [1-\beta_k n\overline{b}_1^2, 1-\beta_k \underline{b}_1^2 w_i[k]]$.
	Define $\overline{\rho}_k = 1 - \epsilon w_i[k]$, then $\overline{\rho}_k$ equals to $1-\epsilon$ or $1$, each with probability $0.5$, and $\expect{\overline{\rho}_k} = 1- \frac{\epsilon}{2} \in (0, 1)$.
	Note that $0 < \epsilon < \frac{\underline{b}_1^2}{n \overline{b}_1^2}$ implies $\overline{\rho}_k > 0$.
	By Lemma \ref{lemma:rv_product}, 
	\begin{align}
		\lim_{k \to \infty} \prod_{i=1}^{k} \overline{\rho}_i = 0,~\text{a.s.}
	\end{align}
	When $\beta_k \in (\frac{\epsilon}{\underline{b}_1^2}, \frac{1}{n\overline{b}_1^2})$, $0 \leq \rho_k \leq \overline{\rho}_k$. Then, in an a.s. sense,
	\begin{align}
		\lim_{k \to \infty} |e[k]| = |e[0]| \lim_{k \to \infty} \prod_{i=1}^{k} \rho_i \leq |e[0]| \lim_{k \to \infty} \prod_{i=1}^{k} \overline{\rho}_i = 0.
	\end{align}
	Since $|e[k]| \geq 0$, $\lim_{k \to \infty} |e[k]| = 0$ a.s.
	In addition, by \eqref{thm1:2}, $\lim_{k \to \infty} \Delta \bm{u}[k] = \zeros_n$ a.s.
\end{proof}
\begin{remark}
	If $\mathcal{U}$ is sufficiently large and no DER hits the capacity limits, then $\bar{\bm{\phi}}[k] = \hat{\bm{\phi}}[k]$ and $\bm{\phi}[k]^\top \bm{W}[k] \bar{\bm{\phi}}[k] \geq \underline{b}_1^2 \sum_{i=1}^n w_i[k]$.
	Following a similar argument as in part (\textit{b}.2) in the proof of Theorem \ref{thm1}, we can show $e[k]$ converges to $0$ a.s. when $\beta_k \in (\frac{\epsilon}{n\underline{b}_1^2}, \frac{1}{n\overline{b}_1^2})$, where $0 < \epsilon < \frac{\underline{b}_1^2}{\overline{b}_1^2}$.
\end{remark}

Following a similar argument, Corollary \ref{corollary:1} can be proved as follows:
\begin{proof}
	When the control update rule in \eqref{eq:control0} is used instead of the one in \eqref{eq:control}, 
	\begin{align} 
		e[k] = (1 - \beta_k \bm{\phi}[k]^\top \bar{\bm{\phi}}[k]) e[k-1].
	\end{align}
	If $\bm{u}[k] \neq \underline{\bm{u}}$ and $\bm{u}[k] \neq \overline{\bm{u}}$, $\bm{\phi}[k]^\top \bar{\bm{\phi}}[k] = \phi_i [k] \bar{\phi}_i[k] \geq \underline{b}_1^2$.
	Define $\rho_k = 1 - \beta_k \bm{\phi}[k]^\top \bar{\bm{\phi}}[k]$.
	When $\beta_k \in (\frac{\epsilon}{\underline{b}_1^2}, \frac{1}{n\overline{b}_1^2})$, $\rho_k < 1-\epsilon$.
	Therefore, $\left|\frac{e[k]}{e[k-1]}\right| = \rho_k < 1- \epsilon$.
\end{proof}

\subsection{Proof of Theorem \ref{thm2}} \label{appendix:thm2}

The convergence analysis of the estimation step uses some convergence results for $\Delta \bm{\phi}[k]$, which are presented next.
\begin{lemma} \label{lemma:rv_sum_prod}
	Let $X_k$, $k=1, 2, \cdots$, be i.i.d. random variables that take value $1$ with probability $0.5$, or some constant $x \in (0, 1)$, also with probability $0.5$. Let $Y_k = \prod_{i=1}^k X_i$ and $Z = \sum_{i=1}^\infty Y_i$. Then, $Z$ is bounded a.s.
\end{lemma}
\begin{proof}
	Let $K$ denote the maximum number of $1$'s that appears continuously in the sequence $\{X_k\}$; then, the sequence $\{Y_k\}$ will have a new (smaller) value at most after $K+1$ steps.
	We claim $Z$ is unbounded only if $K$ is infinite. 
	Suppose $X_j = x$, and $X_k = 1$ for $k = j+1, \cdots, j+m$, then $Y_j = Y_{j+1} = \dots = Y_{j+m}$ and $\sum_{i=j}^{j+m} Y_j = (m+1) Y_j \leq (K+1) Y_j$.
	Therefore,
	\begin{align} \label{eq:l4_1}
		Z = \sum_{i=1}^\infty Y_i \leq (K+1) \sum_{i=0}^\infty x^i = \frac{K+1}{1-x}.
	\end{align}
	It follows from \eqref{eq:l4_1} that $Z$ is unbounded only if $K$ is infinite.
	However, $\prob{M = \infty} \leq \prob{X_{i+1} = \cdots = X_{i+K} = 1,~\text{for some}~i} = \frac{1}{2^\infty} = 0$, where $\bbP$ denotes probability.
	Thus, $Z$ is bounded a.s.
\end{proof}

\begin{lemma} \label{lemma:phi_converge}
	Using estimation update rule \eqref{eq:estimation} and control update rule \eqref{eq:control}, with $\beta_k \in (\frac{\epsilon}{n \underline{b}_1^2}, \frac{1}{n \overline{b}_1^2})$, where $0 < \epsilon < \frac{\underline{b}_1^2}{\overline{b}_1^2}$ is a given parameter, then
		\begin{align} \label{thm2:converge_1}
			\lim_{k \to \infty} \norm{\Delta \bm{\phi}[k]} = 0, ~\text{a.s.}
		\end{align}
		and 
		\begin{align}
			\sum_{k=1}^\infty \norm{\Delta \bm{\phi}[k]} < \infty, \text{a.s.}
		\end{align}
\end{lemma}
\begin{proof} If follows from the proof of Theorem \ref{thm1} that the entries of $\Delta \bm{u}[k]$ always have the same sign. 
	First consider the case where $\Delta \bm{u}[k] \geq \zeros_n$ for all $k \in \bbN$.
	Note that $\bm{\phi}[k]^\top = \left.\dfrac{\partial h}{\partial \bm{u}}\right\vert_{\bm{\tilde{u}}[k]}$, where $\bm{\tilde{u}}[k] = a_k \bm{u}[k] + (1-a_k) \bm{u}[k-1]$ with $a_k \in [0, 1]$, i.e., $\bm{u}[k-1] \leq \bm{\tilde{u}}[k] \leq \bm{u}[k]$.
	Similarly, $\bm{\phi}[k-1]^\top = \left.\dfrac{\partial h}{\partial \bm{u}}\right\vert_{\bm{\tilde{u}}[k-1]}$, where $\bm{u}[k-2] \leq \bm{\tilde{u}}[k-1] \leq \bm{u}[k-1]$.
	Thus, by Assumption \ref{a3}, we have that
	\begin{align}
		\norm{\Delta \bm{\phi}[k]} & \leq b_2 \norm{\bm{\tilde{u}}[k] - \bm{\tilde{u}}[k-1]} \nonumber \\
		& \leq b_2 \norm{\bm{u}[k] - \bm{u}[k-2]} \nonumber \\
		& = b_2 \norm{\Delta \bm{u}[k] + \Delta \bm{u}[k-1]} \nonumber \\
		& \leq b_2 (\norm{\Delta \bm{u}[k]} + \norm{\Delta \bm{u}[k-1]}).
	\end{align}
	Since $\lim_{k \to \infty} \norm{\Delta \bm{u}[k]} = 0$ a.s. by Theorem \ref{thm1}, as a result, $\lim_{k \to \infty} (\norm{\Delta \bm{u}[k]} + \norm{\Delta \bm{u}[k-1]}) = 0$ a.s., which gives
	\begin{align} 
		\lim_{k \to \infty} \norm{\Delta \bm{\phi}[k]} = 0, ~\text{a.s.}
	\end{align}
	Assume $\bm{u}[k]=\zeros_n$ for all $k<0$, then we have that
	\begin{align}
		\sum_{k=1}^\infty \norm{\Delta \bm{\phi}[k]} & \leq \sum_{k=1}^\infty b_2 (\norm{\Delta \bm{u}[k]} + \norm{\Delta \bm{u}[k-1]}) \nonumber \\
		& \leq 2 b_2 \sum_{k=0}^\infty \norm{\Delta \bm{u}[k]} \nonumber \\
		& \leq 2 b_2 \sum_{k=0}^\infty \norm{\beta_k \bm{W}[k] \hat{\bm{\phi}}[k] e[k-1]} \nonumber \\
		& \leq \frac{2 b_2}{n \overline{b}_1^2} \sqrt{n} \overline{b}_1 \sum_{k=0}^\infty |e[k-1]| \nonumber \\
		& = \frac{2 b_2}{\sqrt{n} \overline{b}_1} \sum_{k=-1}^\infty |e[k]|.
	\end{align}
	Recall that $\overline{\rho}_k$ equals to $1-\epsilon$ or $1$, each with probability $0.5$, where $\overline{\rho}_k$ is defined in the proof of Theorem \ref{thm1}.
	Therefore, by Lemma \ref{lemma:rv_sum_prod}, $\sum_{k=1}^\infty \prod_{i=1}^k \overline{\rho}_i$ is bounded  a.s.
	When $\beta_k \in (\frac{\epsilon}{\underline{b}_1^2}, \frac{1}{n\overline{b}_1^2})$, $0 \leq \rho_k \leq \overline{\rho}_k$, and 
	\begin{align}
		\sum_{k=0}^\infty |e[k]| = |e[0]| (1+\sum_{k=1}^\infty\prod_{i=1}^{k} \rho_i) \leq |e[0]| (1 + \sum_{k=1}^\infty \prod_{i=1}^{k} \overline{\rho}_i).
	\end{align}
	As a result, $\sum_{k=-1}^\infty |e[k]|$ is bounded a.s. since $|e[-1]|$ is also bounded.
	The case where $\Delta \bm{u}[k] \leq \zeros_n$ for all $k \in \bbN$ can be proved similarly.
\end{proof}

The convergence analysis of the estimation step also relies on the following lemma (see Theorem 1 in \cite{robbins1985convergence}).
\begin{lemma} \label{lemma:converge}
	Let $X_k, Y_k, Z_k$, $k=1, 2, \cdots$, be non-negative variables in $\bbR$ such that $\sum_{k=0}^\infty Y_k < \infty$, and $X_{k+1} \leq X_k + Y_k - Z_k$, then $X_k$ converges and $\sum_{k=0}^\infty Z_k < \infty$.
\end{lemma}

Using Lemma \ref{lemma:phi_converge} and Lemma \ref{lemma:converge}, Theorem \ref{thm2} can then be proved as follows:
\begin{proof} Consider an arbitrary sample path. 
	Without loss of generality, assume $e[k] < 0$, it follows from Theorem \ref{thm1} that $e[k] < 0, \forall k \in \bbN$.
	Since $\bm{u}[k] \in (\underline{\bm{u}}, \overline{\bm{u}}), \forall k \in \bbN$, \eqref{eq:control} becomes 
	\begin{align}
		\Delta \bm{u}[k] = - \beta_k e[k-1] \bm{W}[k] \hat{\bm{\phi}}[k].
	\end{align}
	It follows from \eqref{eq:io_model} and \eqref{eq:estimation} that	
	\begin{align} \label{thm2:update}
		\hat{\bm{\phi}}[k+1] = \proj_{\calB} ( \hat{\bm{\phi}}[k] - \alpha_{k+1} \Delta \bm{u}[k] \Delta \bm{u}[k]^\top \bm{\varepsilon}[k] ).
	\end{align}
	By definition, the estimation error at iteration $k$ is 
	\begin{align} \label{thm2:error}
		\bm{\varepsilon}[k+1] = \proj_{\calB} ( \hat{\bm{\phi}}[k] - \alpha_{k+1} \Delta \bm{u}[k] \Delta \bm{u}[k]^\top \bm{\varepsilon}[k] ) - \bm{\phi}[k+1].
	\end{align}
	Since $\bm{\phi}[k+1] = \proj_{\calB} (\bm{\phi}[k+1])$, by the non-expansiveness of the projection operation (see Proposition 1.1.9 in \cite{bertsekas2009convex}), then
	\begin{align} \label{thm2:norm_ineq}
		\norm{\bm{\varepsilon}[k+1]} & \leq \norm{ \bm{\varepsilon}[k] - \alpha_{k+1} \Delta \bm{u}[k] \Delta \bm{u}[k]^\top \bm{\varepsilon}[k] - \Delta \bm{\phi}[k+1]} \nonumber \\
		& \leq \norm{ \bm{\varepsilon}[k] - \alpha_{k+1} \Delta \bm{u}[k] \Delta \bm{u}[k]^\top \bm{\varepsilon}[k]} + \norm{\Delta \bm{\phi}[k+1]}.
	\end{align}
	Let $g(\alpha_{k+1}) = \norm{ \bm{\varepsilon}[k] - \alpha_{k+1} \Delta \bm{u}[k] \Delta \bm{u}[k]^\top \bm{\varepsilon}[k] }^2$; then, $g$ attains its minimum at $\alpha_{k+1} = \frac{1}{\norm{\Delta \bm{u}[k]}^2}$, which is 
	\begin{align}
		\norm{\bm{\varepsilon}[k]}^2 - (\bm{\varepsilon}[k]^\top \frac{\Delta \bm{u}[k]}{\norm{\Delta \bm{u}[k]}} )^2 = \norm{\bm{\varepsilon}[k]}^2 -  (\bm{\varepsilon}[k]^\top \frac{\bm{W}[k] \hat{\bm{\phi}}[k]}{ \norm{\bm{W}[k] \hat{\bm{\phi}}[k]} } )^2.
	\end{align}
	Define $\cos \theta_k = \frac{\bm{\varepsilon}[k]^\top}{ \norm{\bm{\varepsilon}[k]}} \frac{\bm{W}[k] \hat{\bm{\phi}}[k]}{\norm{\bm{W}[k] \hat{\bm{\phi}}[k]}}$.
	Consequently, $g(\alpha_{k+1}) = (1-\sin^2 \theta_k) \norm{ \bm{\varepsilon}[k]}^2$, and
	\begin{align}
		\norm{\bm{\varepsilon}[k+1]} \leq |\sin \theta_k| \norm{\bm{\varepsilon}[k]} + \norm{\Delta \bm{\phi}[k+1]}.
	\end{align}
	
	Let $X_k = \norm{\bm{\varepsilon}[k]}$, $Y_k = \norm{\Delta \bm{\phi}[k+1]}$, and $Z_k = (1-|\sin \theta_k|) \norm{\bm{\varepsilon}[k]}$.
	Then, $X_{k+1} \leq X_k + Y_k - Z_k$.
	Also, $\sum_{k=0}^\infty Y_k = \sum_{k=1}^\infty \norm{\Delta \bm{\phi}[k]} < \infty$ by Lemma \ref{lemma:phi_converge}.
	Therefore, by Lemma \ref{lemma:converge}, $\norm{\bm{\varepsilon}[k]}$ converges, and $\sum_{k=1}^\infty (1-|\sin \theta_k|) \norm{\bm{\varepsilon}[k]} < \infty$, which further implies $\lim_{k \to \infty} (1-|\sin \theta_k|) \norm{\bm{\varepsilon}[k]} = 0$.
	Let $\varepsilon^\star$ denote the limit of $\norm{\bm{\varepsilon}[k]}$; then,
	\begin{align}
		\lim_{k \to \infty} |\sin \theta_k| \norm{\bm{\varepsilon}[k]} & = \lim_{k \to \infty} (|\sin \theta_k|-1) \norm{\bm{\varepsilon}[k]} + \lim_{k \to \infty} \norm{\bm{\varepsilon}[k]} \nonumber \\
		&= \varepsilon^\star. 
	\end{align}
	
	Next, we show $\varepsilon^\star = 0$ by contradiction. Assume $\varepsilon^\star > 0$. 
	Since both $\norm{\bm{\varepsilon}[k]}$ and $|\sin \theta_k| \norm{\bm{\varepsilon}[k]}$ converges to $\varepsilon^\star$,
	\begin{align}
		\lim_{k\to \infty} |\sin \theta_k| = \frac{\lim_{k\to \infty} |\sin \theta_k| \norm{\bm{\varepsilon}[k]}}{\lim_{k\to \infty} \norm{\bm{\varepsilon}[k]}} = 1,
	\end{align}
	which implies $|\cos \theta_k|$ converges to $0$.
	Since $\norm{\bm{\varepsilon}[k]}$ and $\norm{\bm{W}[k] \hat{\bm{\phi}}[k]}$ are bounded, then $|\bm{\varepsilon}[k]^\top \bm{W}[k] \hat{\bm{\phi}}[k]|$ converges to $0$.
	Define $\mathsf{E}_i[k] = \{w_j[k] = 1~\text{if } j=i, w_j[k] = 0~\text{otherwise}\}$; then $\prob{\mathsf{E}_i[k]} = \frac{1}{2^n}$. 
	Consequently, $\sum_{k=1}^\infty \prob{\mathsf{E}_i[k]} = \infty$.
	Also note that $\mathsf{E}_i[k]$, $k \in \bbN$, are independent.
	By the Borel-Cantelli Lemma (see Lemma 1.3 in \cite{hajek2015random}), $\prob{\mathsf{E}_i[k]~\text{infinitely often}} = 1$; therefore, there are infinitely many time instants that $w_i[k] = 1$ and $w_j[k] = 0$ for all $j \neq i$.
	Let $\mathcal{K}_i$ denote the set of such time instants.
	Then $|\bm{\varepsilon}[k]^\top \bm{W}[k] \hat{\bm{\phi}}[k]| = |\varepsilon_i[k] \hat{\phi}_i[k]|$ for $k \in \mathcal{K}_i$. 
	The sequence $\{|\varepsilon_i[k] \hat{\phi}_i[k]|, k \in \mathcal{K}_i\}$ is a subsequence of $\{|\bm{\varepsilon}[k]^\top \bm{W}[k] \hat{\bm{\phi}}[k]|\}$; therefore, it also converges to $0$.
	Note that $\hat{\bm{\phi}}[k] > 0$; thus, $\varepsilon_i[k]$ converges to $0$.
	Since $i$ is arbitrary, we conclude that $\norm{\bm{\varepsilon}}[k]$ converges to $0$, which implies $\varepsilon^\star = 0$, contradiction.
	Since this result holds for all sample paths, then we conclude that $\norm{\bm{\varepsilon}[k]}$ converges to $0$ a.s.
\end{proof}

\bibliographystyle{IEEEtran}
\bibliography{ARC}

\begin{thebibliography}{10}
\providecommand{\url}[1]{#1}
\csname url@samestyle\endcsname
\providecommand{\newblock}{\relax}
\providecommand{\bibinfo}[2]{#2}
\providecommand{\BIBentrySTDinterwordspacing}{\spaceskip=0pt\relax}
\providecommand{\BIBentryALTinterwordstretchfactor}{4}
\providecommand{\BIBentryALTinterwordspacing}{\spaceskip=\fontdimen2\font plus
\BIBentryALTinterwordstretchfactor\fontdimen3\font minus
  \fontdimen4\font\relax}
\providecommand{\BIBforeignlanguage}[2]{{%
\expandafter\ifx\csname l@#1\endcsname\relax
\typeout{** WARNING: IEEEtran.bst: No hyphenation pattern has been}%
\typeout{** loaded for the language `#1'. Using the pattern for}%
\typeout{** the default language instead.}%
\else
\language=\csname l@#1\endcsname
\fi
#2}}
\providecommand{\BIBdecl}{\relax}
\BIBdecl

\bibitem{driesen2008design}
J.~Driesen and F.~Katiraei, ``Design for distributed energy resources,''
  \emph{IEEE Power Energy Mag.}, vol.~6, no.~3, pp. 30--40, May 2008.

\bibitem{kulmala2014voltage}
A.~Kulmala, S.~Repo, and P.~J\"{a}rventausta, ``Coordinated voltage control in
  distribution networks including several distributed energy resources,''
  \emph{IEEE Trans. Smart Grid}, vol.~5, no.~4, pp. 2010--2020, July 2014.

\bibitem{xu2018voltage}
H.~Xu, A.~D. Dom{\'\i}nguez-Garc{\'\i}a, and P.~W. Sauer, ``A data-driven
  voltage control framework for power distribution systems,'' in \emph{Proc. of
  IEEE PES General Meeting}, Portland, OR, Aug. 2018, pp. 1--5.

\bibitem{zhang2018dynamic}
K.~Zhang, W.~Shi, H.~Zhu, E.~Dall'Anese, and T.~Basar, ``Dynamic power
  distribution system management with a locally connected communication
  network,'' \emph{IEEE J. Sel. Topics Signal Process.}, vol.~12, no.~4, pp.
  673--687, Aug 2018.

\bibitem{adg2010der}
A.~D. Dom\'{i}nguez-Garc\'{i}a and C.~N. Hadjicostis, ``Coordination and
  control of distributed energy resources for provision of ancillary
  services,'' in \emph{Proc. of IEEE International Conference on Smart Grid
  Communications}, Oct. 2010, pp. 537--542.

\bibitem{guggilam2017regulation}
E.~Dall'Anese, S.~Guggilam, A.~Simonetto, Y.~C. Chen, and S.~V. Dhople,
  ``Optimal regulation of virtual power plants,'' \emph{IEEE Trans. Power
  Syst.}, vol.~33, no.~2, pp. 1868--1881, March 2018.

\bibitem{xu2018frequency}
H.~Xu, A.~D. Dom{\'\i}nguez-Garc{\'\i}a, and P.~W. Sauer, ``Adaptive
  coordination of distributed energy resources in lossy power distribution
  systems,'' in \emph{Proc. of IEEE PES General Meeting}, Portland, OR, Aug.
  2018, pp. 1--5.

\bibitem{ardakanian2017identification}
O.~Ardakanian, V.~W. Wong, R.~Dobbe, S.~H. Low, A.~von Meier, C.~Tomlin, and
  Y.~Yuan, ``On identification of distribution grids,'' \emph{arXiv preprint
  arXiv:1711.01526}, 2017.

\bibitem{chen2014sensitivity}
Y.~C. Chen, A.~D. Dom\'inguez-Garc\'ia, and P.~W. Sauer, ``Measurement-based
  estimation of linear sensitivity distribution factors and applications,''
  \emph{IEEE Trans. Power Syst.}, vol.~29, no.~3, pp. 1372--1382, May 2014.

\bibitem{chen2016jacobian}
Y.~C. Chen, J.~Wang, A.~D. Dom\'inguez-Garc\'ia, and P.~W. Sauer,
  ``Measurement-based estimation of the power flow jacobian matrix,''
  \emph{IEEE Trans. Smart Grid}, vol.~7, no.~5, pp. 2507--2515, Sept. 2016.

\bibitem{hou2011data}
Z.~Hou and S.~Jin, ``Data-driven model-free adaptive control for a class of
  mimo nonlinear discrete-time systems,'' \emph{IEEE Trans. Neural Netw.},
  vol.~22, no.~12, pp. 2173--2188, Dec. 2011.

\bibitem{zhang2016sensitivity}
J.~Zhang, X.~Zheng, Z.~Wang, L.~Guan, and C.~Y. Chung, ``Power system
  sensitivity identification -- inherent system properties and data quality,''
  \emph{IEEE Trans. Power Syst.}, vol.~32, no.~4, pp. 2756--2766, July 2017.

\bibitem{chen2015sparse}
Y.~C. Chen, A.~D. Dom\'inguez-Garc\'ia, and \hspace{0cm} P.~W.~Sauer, ``A
  sparse representation approach to online estimation of power system
  distribution factors,'' \emph{IEEE Transactions on Power Systems}, vol.~30,
  no.~4, pp. 1727--1738, July 2015.

\bibitem{horn2016sced}
K.~E.~V. Horn, A.~D. Dom\'inguez-Garc\'ia, and P.~W. Sauer, ``Measurement-based
  real-time security-constrained economic dispatch,'' \emph{IEEE Trans. Power
  Syst.}, vol.~31, no.~5, pp. 3548--3560, Sept. 2016.

\bibitem{lu2015wams}
C.~Lu, Y.~Zhao, K.~Men, L.~Tu, and Y.~Han, ``Wide-area power system stabiliser
  based on model-free adaptive control,'' \emph{IET Control Theory Appl.},
  vol.~9, no.~13, pp. 1996--2007, 2015.

\bibitem{zhang2016damping}
J.~Zhang, C.~Y. Chung, and Y.~Han, ``Online damping ratio prediction using
  locally weighted linear regression,'' \emph{IEEE Trans. Power Syst.},
  vol.~31, no.~3, pp. 1954--1962, May 2016.

\bibitem{hou2017overview}
Z.~Hou, R.~Chi, and H.~Gao, ``An overview of dynamic-linearization-based
  data-driven control and applications,'' \emph{IEEE Trans. Ind. Electron.},
  vol.~64, no.~5, pp. 4076--4090, May 2017.

\bibitem{zhang2017sensitivity}
J.~Zhang, Z.~Wang, X.~Zheng, L.~Guan, and C.~Y. Chung, ``Locally weighted ridge
  regression for power system online sensitivity identification considering
  data collinearity,'' \emph{IEEE Trans. Power Syst.}, vol.~PP, no.~99, pp.
  1--1, 2017.

\bibitem{zhang2017noise}
J.~Zhang, C.~Y. Chung, and L.~Guan, ``Noise effect and noise-assisted ensemble
  regression in power system online sensitivity identification,'' \emph{IEEE
  Trans. Ind. Informat.}, vol.~13, no.~5, pp. 2302--2310, Oct. 2017.

\bibitem{zhang2016mpm}
J.~Zhang and H.~\vspace{0in}Xu, ``Microperturbation method for power system
  online model identification,'' \emph{IEEE Trans. Ind. Informat.}, vol.~12,
  no.~3, pp. 1055--1063, June 2016.

\bibitem{zhang2017online}
J.~Zhang and H.~Xu, ``Online identification of power system equivalent inertia
  constant,'' \emph{IEEE Trans. Ind. Electron.}, vol.~64, no.~10, pp.
  8098--8107, Oct 2017.

\bibitem{iusem2003convergence}
A.~Iusem, ``On the convergence properties of the projected gradient method for
  convex optimization,'' \emph{Computational \& Applied Mathematics}, vol.~22,
  no.~1, pp. 37--52, 2003.

\bibitem{ajala2017hierarchy}
O.~Ajala, M.~Almeida, P.~W. I.~Celanovic, Sauer, and A.~D.
  Dom{\'\i}nguez-Garc{\'\i}a, ``A hierarchy of models for microgrids with
  grid-feeding inverters,'' in \emph{Proc. of the IREP Bulk Power System
  Dynamics and Control Symposium}, Aug. 2017.

\bibitem{test_feeder}
\BIBentryALTinterwordspacing
{IEEE} distribution test feeders. [Online]. Available:
  \url{https://ewh.ieee.org/soc/pes/dsacom/testfeeders/}
\BIBentrySTDinterwordspacing

\bibitem{matpower}
R.~D. Zimmerman, C.~E. Murillo-Sanchez, and R.~J. Thomas, ``Matpower:
  Steady-state operations, planning, and analysis tools for power systems
  research and education,'' \emph{IEEE Trans. Power Syst.}, vol.~26, no.~1, pp.
  12--19, Feb. 2011.

\bibitem{robbins2016optimal}
B.~A. Robbins and A.~D. Dom{\'\i}nguez-Garc{\'\i}a, ``Optimal reactive power
  dispatch for voltage regulation in unbalanced distribution systems,''
  \emph{IEEE Trans. Power Syst.}, vol.~31, no.~4, pp. 2903--2913, 2016.

\bibitem{hajek2015random}
B.~Hajek, \emph{Random processes for engineers}.\hskip 1em plus 0.5em minus
  0.4em\relax Cambridge University Press, 2015.

\bibitem{robbins1985convergence}
H.~Robbins and D.~Siegmund, ``A convergence theorem for non negative almost
  supermartingales and some applications,'' in \emph{Herbert Robbins Selected
  Papers}.\hskip 1em plus 0.5em minus 0.4em\relax Springer, 1985, pp. 111--135.

\bibitem{bertsekas2009convex}
D.~P. Bertsekas, \emph{Convex optimization theory}.\hskip 1em plus 0.5em minus
  0.4em\relax Athena Scientific Belmont, 2009.

\end{thebibliography}



\vspace{0.5in}
\begin{IEEEbiographynophoto}{Hanchen Xu}
(S'12) received the B.Eng. and M.S. degrees in electrical engineering from Tsinghua University, Beijing, China, in 2012 and 2014, respectively, and the M.S. degree in applied mathematics from the University of Illinois at Urbana-Champaign, Urbana, IL, USA, in 2017, where he is currently working toward the Ph.D. degree at the Department of Electrical and Computer Engineering. His current research interests include control, optimization, reinforcement learning, with applications to power systems and electricity markets.
\end{IEEEbiographynophoto}

\begin{IEEEbiographynophoto}{Alejandro~D.~Dom\'{i}nguez-Garc\'{i}a}
(S'02, M'07) received the degree of electrical engineering from the University of Oviedo (Spain) in 2001 and the Ph.D. degree in electrical engineering and computer science from the Massachusetts Institute of Technology, Cambridge, MA, in 2007. 

He is Professor with the Department of Electrical and Computer Engineering (ECE), and Research Professor with the Coordinated Science Laboratory and the Information Trust Institute, all at the University of Illinois at Urbana-Champaign. He is affiliated with the ECE Power and Energy Systems area, and has been a Grainger Associate since August 2011.

His research interests are in the areas of system reliability theory and control, and their applications to electric power systems, power electronics, and embedded electronic systems for safety-critical/fault-tolerant aircraft, aerospace, and automotive applications.

Dr. Dom\'{i}nguez-Garc\'{i}a received the NSF CAREER Award in 2010, and the Young Engineer Award from the IEEE Power and Energy Society in 2012. In 2014, he was invited by the National Academy of Engineering to attend the US Frontiers of Engineering Symposium, and was selected by the University of Illinois at Urbana-Champaign Provost to receive a Distinguished Promotion Award. In 2015, he received the U of I College of Engineering Dean's Award for Excellence in Research.

He is currently an editor for the IEEE Transactions on Control of Network Systems; he also served as an editor of the IEEE Transactions on Power Systems and IEEE Power Engineering Letters from 2011 to 2017.
 
\end{IEEEbiographynophoto}

\begin{IEEEbiographynophoto}{Peter W. Sauer}
(S'73, M'77, SM'82, F'93, LF'12) obtained his Bachelor of Science degree in electrical engineering from the University of Missouri at Rolla in 1969. From 1969 to 1973, he was the electrical engineer on a design assistance team for the Tactical Air Command at Langley Air Force Base, Virginia. He obtained the Master of Science and Ph.D. degrees in Electrical Engineering from Purdue University in 1974 and 1977 respectively. From August 1991 to August 1992 he served as the Program Director for Power Systems in the Electrical and Communication Systems Division of the National Science Foundation in Washington D.C. He is a cofounder of the Power Systems Engineering Research Center (PSERC) and the PowerWorld Corporation. He is a registered Professional Engineer in Virginia and Illinois, a Fellow of the IEEE, and a member of the U.S. National Academy of Engineering. He is currently the Grainger Chair Professor of Electrical Engineering at Illinois. Additional information can be found at: https://ece.illinois.edu/directory/profile/psauer.
\end{IEEEbiographynophoto}

\end{document}